\documentclass[a4paper]{amsart}

\usepackage[utf8]{inputenc}
\usepackage[T1]{fontenc}
\usepackage{lmodern, enumerate}
\usepackage{amssymb,amsxtra,amscd,amsmath}
\usepackage{amsthm}
\usepackage{stmaryrd}
\usepackage{amsfonts}
\usepackage{tikz-cd}
\usepackage{bbm}
\usepackage[all]{xy}
\usepackage{xcolor}
\usepackage{nicefrac,mathtools}
\usepackage[inline]{enumitem}
\usepackage{microtype}
\usepackage{comment}
\usepackage{csquotes}
\usepackage[
  pdftitle={A groupoid model for Rordam's finite--infinite C*-algebra},
  pdfauthor={Diego Mart\'{i}nez},
  pdfsubject={Mathematics},
  colorlinks=true, linkcolor=blue, linkbordercolor=blue, citecolor=red, citebordercolor=red, urlcolor=blue, linktocpage=true
]{hyperref}
\usepackage[capitalise, nameinlink, noabbrev, nosort]{cleveref} 
\usepackage[lite]{amsrefs}

\newtheorem{alphthm}{Theorem}     

\newtheorem{alphcor}[alphthm]{Corollary}           

\setlist[enumerate]{font=\normalfont}
\crefname{enumi}{}{} 
\crefname{enumi}{}{} 
\setlist[enumerate]{label=(\roman*)}

\numberwithin{equation}{section}
\theoremstyle{plain}
\newtheorem{theorem}[equation]{Theorem}
\newtheorem{lemma}[equation]{Lemma}
\newtheorem{proposition}[equation]{Proposition}
\newtheorem{corollary}[equation]{Corollary}

\theoremstyle{definition}
\newtheorem{definition}[equation]{Definition}
\theoremstyle{remark}
\newtheorem{remark}[equation]{Remark}

\newtheorem{notation}[equation]{Notation}

\newtheorem{question}[equation]{Question}
\newtheorem{convention}[equation]{Convention}
\newtheorem{assumption}[equation]{Assumption}

\crefname{theorem}{Theorem}{Theorems}
\crefname{alphthm}{Theorem}{Theorems}
\crefname{proposition}{Proposition}{Propositions}
\crefname{lemma}{Lemma}{Lemmas}
\crefname{claim}{Claim}{Claims}
\crefname{remark}{Remark}{Remarks}
\crefname{conjecture}{Conjecture}{Conjectures}
\crefname{corollary}{Corollary}{Corollaries}
\crefname{question}{Question}{Questions}
\crefname{conjecture}{Conjecture}{Conjectures}
\crefname{fact}{Fact}{Facts}
\crefname{claim}{Claim}{Claims}
\crefname{case}{Case}{Cases}
\crefname{convention}{Convention}{Conventions}

\calclayout

\NewDocumentCommand{\CC}{}{\mathbb{C}}

             \NewDocumentCommand{\NN}{}{\mathbb{N}}

            \NewDocumentCommand{\TT}{}{\mathbb{T}}

             \NewDocumentCommand{\ZZ}{}{\mathbb{Z}}

\NewDocumentCommand{\CA}{}{\mathcal{A}}            \NewDocumentCommand{\CB}{}{\mathcal{B}}
\NewDocumentCommand{\CCC}{}{\mathcal{C}}           
            \NewDocumentCommand{\CF}{}{\mathcal{F}}
\NewDocumentCommand{\CG}{}{\mathcal{G}}            \NewDocumentCommand{\CH}{}{\mathcal{H}}
            
\NewDocumentCommand{\CK}{}{\mathcal{K}}            
\NewDocumentCommand{\CM}{}{\mathcal{M}}            
\NewDocumentCommand{\CO}{}{\mathcal{O}}

\NewDocumentCommand{\CU}{}{\mathcal{U}}            \NewDocumentCommand{\CV}{}{\mathcal{V}}
            
            \NewDocumentCommand{\CZ}{}{\mathcal{Z}}

\NewDocumentCommand{\fa}{}{\mathfrak{a}}            \NewDocumentCommand{\fb}{}{\mathfrak{b}}
\NewDocumentCommand{\fc}{}{\mathfrak{c}}            \NewDocumentCommand{\fd}{}{\mathfrak{d}}

\NewDocumentCommand{\fs}{}{\mathfrak{s}}

\newcommand{\CGo}{\mathcal{G}^{(0)}}

\NewDocumentCommand{\Egrp}{omm}{%
  \IfNoValueTF{#1}
    {E\left(#2,#3\right)}
    {E_{#1}\left(#2,#3\right)}
}

\NewDocumentCommand{\KKgrp}{omm}{%
  \IfNoValueTF{#1}
    {KK\left(#2,#3\right)}
    {KK_{#1}\left(#2,#3\right)}
}
\NewDocumentCommand{\Kgrp}{om}{%
  \IfNoValueTF{#1}
    {K_\star\left(#2\right)}
    {K_{#1}\left(#2\right)}
}

\NewDocumentCommand{\variable}{}{-}

\NewDocumentCommand{\diagelms}{}{C_{{\sigma}}\left(Z, \ell^\infty\left(\NN\right)\right)}

\NewDocumentCommand{\hkindalg}{om}{%
  \IfNoValueTF{#1}
    {\CCC_0\left(#2\right)}
    {\CCC_{0,#1}\left(#2\right)}
}
\NewDocumentCommand{\hkalg}{om}{%
  \IfNoValueTF{#1}
    {\CA\left(#2\right)}
    {\CA_{ {#1} }\left(#2\right)}
}

\NewDocumentCommand{\hkhilb}{om}{%
  \IfNoValueTF{#1}
    {\CH(#2)}
    {\CH_{#1}(#2)}
}

\NewDocumentCommand{\fhalg}{O{h}m}{%
  \IfNoValueTF{#1}
    {\CF\left(#2\right)}
    {\CF_{#1}\left(#2\right)}
}
\NewDocumentCommand{\schwartzfuncs}{om}{%
  \IfNoValueTF{#1}
    {\fs\left(#2\right)}
    {\fs\left(#1,#2\right)}
}
\NewDocumentCommand{\delmaux}{}{\alpha}
\NewDocumentCommand{\ddelmaux}{}{\beta}
\NewDocumentCommand{\delm}{o}{%
  \IfNoValueTF{#1}
    {\delmaux}
    {\delmaux^{#1}}
}
\NewDocumentCommand{\reddelm}{o}{%
  \IfNoValueTF{#1}
    {\delmaux_{\red}}
    {\delmaux^{#1}_{\red}}
}
\NewDocumentCommand{\maxdelm}{o}{%
  \IfNoValueTF{#1}
    {\delmaux_{{\rm max}}}
    {\delmaux^{#1}_{{\rm max}}}
}
\NewDocumentCommand{\ddelm}{o}{%
  \IfNoValueTF{#1}
    {\ddelmaux}
    {\ddelmaux^{#1}}
}
\NewDocumentCommand{\redddelm}{o}{%
  \IfNoValueTF{#1}
    {\ddelmaux_{\red}}
    {\ddelmaux^{#1}_{\red}}
}
\NewDocumentCommand{\maxddelm}{o}{%
  \IfNoValueTF{#1}
    {\ddelmaux_{{\rm max}}}
    {\ddelmaux^{#1}_{{\rm max}}}
}

\NewDocumentCommand{\crossedprod}{omm}{%
  \IfNoValueTF{#1}
    {#2 \rtimes #3}
    {#2 \rtimes_{#1} #3}
}
\NewDocumentCommand{\algcrossedprod}{omm}{%
  \IfNoValueTF{#1}
    {#2 \rtimes_{{\rm alg}} #3}
    {#2 \rtimes^{#1}_{{\rm alg}} #3}
}
\NewDocumentCommand{\redcrossedprod}{omm}{%
  \IfNoValueTF{#1}
    {#2 \rtimes_{\red} #3}
    {#2 \rtimes^{#1}_{\red} #3}
}
\NewDocumentCommand{\maxcrossedprod}{omm}{%
  \IfNoValueTF{#1}
    {#2 \rtimes_{\full} #3}
    {#2 \rtimes^{#1}_{\full} #3}
}

\NewDocumentCommand{\cstar}{}{\texorpdfstring{\(C^*\)\nobreakdash-\hspace{0pt}}{*-}}
\NewDocumentCommand{\Star}{}{\texorpdfstring{\(^*\)\nobreakdash-\hspace{0pt}}{*-}}

\NewDocumentCommand{\norm}{m}{\left\|#1\right\|}

\NewDocumentCommand{\abs}{m}{\left|#1\right|}
\NewDocumentCommand{\innerprod}{mm}{\langle #1, #2 \rangle}
\NewDocumentCommand{\linnerprod}{omm}{%
  \IfNoValueTF{#1}
    {_{\rm L}\innerprod{#2}{#3}}
    {_{#1}\innerprod{#2}{#3}}
}
\NewDocumentCommand{\rinnerprod}{omm}{%
  \IfNoValueTF{#1}
    {\innerprod{#2}{#3}_{\rm R}}
    {\innerprod{#2}{#3}_{#1}}
}

\NewDocumentCommand{\red}{}{\rm red}
\NewDocumentCommand{\full}{}{\rm max}

\NewDocumentCommand{\alg}{}{\rm alg}
\NewDocumentCommand{\redalg}{om}{%
  \IfNoValueTF{#1}
    {C^*_{\red}\left(#2\right)}
    {C^*_{\red,#1}\left(#2\right)}
}
\NewDocumentCommand{\redalgnopar}{om}{%
  \IfNoValueTF{#1}
    {C^*_{\red}(#2)}
    {C^*_{\red,#1}(#2)}
}
\NewDocumentCommand{\maxalg}{om}{%
  \IfNoValueTF{#1}
    {C^*_{\full}\left(#2\right)}
    {C^*_{\full,#1}\left(#2\right)}
}             
\NewDocumentCommand{\maxalgnopar}{om}{%
  \IfNoValueTF{#1}
    {C^*_{\full}(#2)}
    {C^*_{\full,#1}(#2)}
}

\NewDocumentCommand{\algalg}{om}{%
  \IfNoValueTF{#1}
    {\CC_{\alg}\left(#2\right)}
    {\CC_{\alg,#1}\left(#2\right)}
}
\NewDocumentCommand{\algalgnopar}{om}{%
  \IfNoValueTF{#1}
    {\CC_{\alg}(#2)}
    {\CC_{\alg,#1}(#2)}
}

\NewDocumentCommand{\singideal}{om}{%
  \IfNoValueTF{#1}
    {J_{#2}}
    {J_{#2}^{#1}}
}
\NewDocumentCommand{\condexp}{o}{%
  \IfNoValueTF{#1}
    {E}
    {E_{#1}}
}
\NewDocumentCommand{\redcondexp}{o}{%
  \IfNoValueTF{#1}
    {P}
    {P_{#1}}
}
\NewDocumentCommand{\esscondexp}{o}{%
  \IfNoValueTF{#1}
    {EL}
    {EL_{#1}}
}

\NewDocumentCommand{\tensor}{}{\otimes}

\NewDocumentCommand{\id}{o}{%
  \IfNoValueTF{#1}
    {{\rm id}}
    {{\rm id}_{#1}}
}

\begin{document}

\title[A groupoid model for R\o rdam's finite--infinite \cstar{}algebra]{A groupoid model for R\o rdam's finite--infinite \cstar{}algebra}

\author[Drieke Keuppens]{Drieke Keuppens \(^{1}\)}
\address{Department of Mathematics, KU Leuven, Celestijnenlaan 200B, 3001 Leuven, Belgium.}
\email{drieke.keuppens@student.kuleuven.be}

\author[Diego Mart\'{i}nez]{Diego Mart\'{i}nez \(^{2}\)}
\address{Department of Mathematics, KU Leuven, Celestijnenlaan 200B, 3001 Leuven, Belgium.}
\email{diego.martinez@kuleuven.be}

\begin{abstract}
  We construct a groupoid model for R\o rdam's example of a simple, nuclear, separable \cstar{}algebra containing a finite and an infinite projection.
  This groupoid model arises from a (crossed product of a) carefully chosen model for the inductive limit construction in R\o rdam's original approach, which we show to yield a Cartan subalgebra of the limit.
  In particular, this procedure yields a locally compact, second countable, Hausdorff, \'etale, amenable, minimal and topologically principal groupoid whose unit space is compact, and whose \cstar{}algebra contains an infinite and a non-zero finite projection.
  Hence, this groupoid is amenable, and yet does not have comparison.
\end{abstract}



\thanks{{\(^{2}\)} Supported by projects G085020N and 1218726N funded by the Research Foundation Flanders (FWO)}

\maketitle

\section{Introduction} \label[section]{sec:intro}

Within the study of Operator Algebras there are (at least) two different major research currents.
On the one hand, an ``abstract'' approach tackles the study, classification and defining properties of arbitrary \cstar{}algebras and their morphisms (usually \Star{}homomorphisms).
On the other, a more ``concrete'' approach addresses the construction and possible descriptions of different classes of \cstar{}algebras and their morphisms.
Typically speaking, the \emph{Elliott Classification Program} (see \cite{TikuisisAaron2017QonC} and references therein) is contained in the former, while \emph{groupoids} and \emph{Cartan pairs} belong to the latter, see \cites{Renault2008CartanSI,sims2018etalegroupoidscalgebras} among many others.
In this paper we construct a groupoid model for R\o rdam's ``pathological'' \cstar{}algebra \(D \rtimes_\alpha \ZZ\) in \cite{rordamsalgebra}, hence bringing these two research currents closer.

\vspace{1mm}

\emph{Cartan subalgebras} were first introduced in the 1970's for von Neumann algebras by Feldman and Moore \cite{feldman-moore-1997-ii}. They defined a \emph{Cartan subalgebra} $B$ of $A$ as a maximal abelian subalgebra with a faithful conditional expectation from $A$ to $B$ and satisfying some regularity condition.
Translating from the von Neumann to the \cstar{}setting, Kumjian introduced \emph{\cstar{}diagonals} of \cstar{}algebras in \cite{kumjian-diagonals-1986}.
These are pairs with the \emph{unique extension property} (and some other properties), meaning that every pure state on a \cstar{}diagonal has a unique extension to a pure state on the ambient \cstar{}algebra.
Kumjian also proves that the existence of a $C^*$-diagonal corresponds to a \emph{twisted groupoid} model where the groupoid is \emph{principal} (see \cite{kumjian-diagonals-1986} or \cref{sec:preliminaries} below).
This latter reconstruction result was then generalized by Renault, who defined \emph{Cartan subalgebras} of \cstar{}algebras in a similar way to Kumjian, but without the unique extension property.
In this way, Renault also proves in \cite{Renault2008CartanSI} that there is a one-to-one correspondence between twists over groupoids and Cartan pairs. Having this result of Renault, a natural question to ask is the following. 
  
\begin{question}\label[question]{question: intro 1}
  Which \cstar{}algebras admit Cartan subalgebras, and thus groupoid models?
\end{question}
  
As it turns out, every \emph{classifiable} \cstar{}algebra (in the sense of the \emph{Elliott Classification Program}, more on this later) actually has a groupoid model, see \cite{li-classification-2020}.
This notwithstanding, there are examples of \cstar{}algebras without Cartan subalgebras. For instance, \cite{li_cartan_2017} proves $C^*_\lambda(\mathbb{F}_2)$ has no Cartan subalgebras, based on the analogue result of Voiculescu \cite{Voiculescu1996} for the von Neumann algebra $L(\mathbb{F}_2)$.

\vspace{1mm}

As mentioned, in this paper we find a groupoid model for the \cstar{}algebra $D\rtimes_\alpha \ZZ$ constructed by R\o rdam in \cite{rordamsalgebra}.
In fact, this \cstar{}algebra plays an important role in the Classification of \cstar{}algebras.
Classification of $C^*$-algebras first gained interest in the 1970's when Elliott wrote his seminal \cite{ELLIOTT197629}, which gives a set of invariants completely classifying \emph{AF algebras}.
In order to do this he used their \emph{K-theory}, \emph{i.e.}\ two abelian groups $K_0(A)$ and $K_1(A)$ associated to a \cstar{}algebra $A$.
Eventually, and after intense work of many people, this led to \emph{Elliott's Classification Program}, which aimed to classify all unital, simple, separable, nuclear \cstar{}algebras.
Later, in the 1990's, another important development was made by Kirchberg and Phillips in \cites{Kirchberg1994,phillips1997classificationtheoremnuclearpurely}, where they prove a classification result for \emph{Kirchberg algebras}, which are separable, simple, nuclear, purely infinite \cstar{}algebras.
The prototypical examples of these are the Cuntz algebras \(\CO_n\) from \cite{Cuntz:1977ut}.
Furthermore, for technical reasons, Kirchberg and Phillips added an extra condition: they only classified Kirchberg algebras satisfying the \emph{UCT}, or \emph{Universal Coefficient Theorem} \cite{rosenberg-schochet-1987-uct}.
This is a technical condition requiring the existence and exactness of some sequence involving K-theory and \emph{KK-theory}~\cites{TikuisisAaron2017QonC,Kasparov-1988}.
Even though this is a topic of great interest, for the purposes of this text the UCT plays essentially no role (it does play a role in \cite{rordamsalgebra} though).

Combining the results of Elliott and Kirchberg-Phillips gives a classification for some \emph{finite} and ``all'' \emph{purely infinite} \cstar{}algebras.
More formally, it gives a classification for all purely infinite \cstar{}algebras that are simple, separable, unital, nuclear and satisfy the UCT, and for some stably finite \cstar{}algebras satisfying these same conditions. Naturally, this led to the following question. 
  
\begin{question}\label[question]{question: intro 2}
  Is every nuclear, simple, unital, separable \cstar{}algebra satisfying the UCT either stably finite or purely infinite?
\end{question}

Interestingly, the answer is ``no''.
This was shown by R\o rdam in \cite{rordamsalgebra}, where he constructed the \cstar{}algebra $D\rtimes_\alpha\ZZ$ and proved that it has both an infinite and a nonzero finite projection.\footnote{\, Technically speaking, \(D \rtimes_\alpha \ZZ\) does not answer \cref{question: intro 2} negatively, for it is not unital. However, a certain corner of it, which we will denote by \(P(D \rtimes_\alpha \ZZ)P\), does.}
With this counterexample R\o rdam showed that Elliott's Conjecture does not hold in its full generality, see \cite{rordamsalgebra}*{Corollary~7.9}.
Eventually, however, this classification was fully realized by work of ``many hands'', see \cite{TikuisisAaron2017QonC}*{Corollary D}.
The needed ingredient was an extra condition which, when imposed on the \cstar{}algebras, would exclude counterexamples like that of R\o rdam's.
Hence, in order to be \textit{classifiable}, a \cstar{}algebra needs to be $\mathcal{Z}$-stable, where $\mathcal{Z}$ is the \emph{Jiang-Su algebra} of \cite{jiang_simple_1999}.
Notice that R\o rdam's algebra is not \(\CZ\)-stable \cite{rordamsalgebra}*{Corollary 7.5}, since all classifiable \cstar{}algebras are either stably finite or purely infinite.
The key feature of \(D \rtimes_\alpha \ZZ\) is precisely the aforementioned finite-vs-infinite behavior which, as it turns out, is at odds with \(\CZ\)-stability and with \emph{strict comparison} (more on this later).

\vspace{1mm}

In this paper we translate the construction of R\o rdam in \cite{rordamsalgebra} to the more concrete setting of groupoid \cstar{}algebras.
Particularly, we are interested in mirroring the finite-vs-infinite type behavior displayed by \(D \rtimes_\alpha \ZZ\) in a purely dynamical setting.
The following is hence the main result of the paper. We express it in terms of groupoids, but notice that, in order to show it we in fact do prove that R\o rdam's algebra \(D \rtimes_\alpha \ZZ\) in \cite{rordamsalgebra} has a well-behaved Cartan subalgebra whose structure we describe, see \cref{cor:cartan-model}.

\begin{alphthm}[see \cref{cor:thm-intro}] \label[theorem]{thm-intro}
    There is a locally compact, \'etale, Hausdorff, second countable, amenable, minimal, and topologically principal groupoid \(\CH\) with compact unit space such that the unit \(1 \in C^*_r(\CH)\) is infinite, but \(C^*_r(\CH)\) also contains a non-zero finite projection.
\end{alphthm}

\cref{thm-intro} answers \cref{question: intro 1} positively for the pathological algebra answering \cref{question: intro 2} negatively.
As mentioned, a motivation to study \cref{thm-intro} comes from the relationship between \emph{strict comparison} and \emph{dynamical comparison}.
These are two different ways to ``compare'' positive elements: the former is in the \cstar{}algebra, while the latter is directly in the dynamical system.

Recall that a simple, separable, unital, nuclear \cstar{}algebra \(A\) is said to have \emph{strict comparison} (\emph{of positive elements}) if for all \(a, b \in (A \tensor \CK)_+\), if \(d_\tau(a) < d_\tau(b)\) for all tracial states \(\tau \colon A \to \CC\) implies there are \(\{v_n\}_{n \in \NN} \subseteq A \tensor \CK\) such that \(\norm{a - v_n b v_n^*} \to 0\), \emph{i.e.}\ \(a\) is \emph{Cuntz-subequivalent} to \(b\).\footnote{\, It is not relevant for our purposes, but \(d_\tau(a) \coloneqq \lim_{n \to\infty} \tau(a^{1/n})\) is the so-called rank function associated to \(\tau\). This is a way to measure the size of \(a\) regardless of the values of \(a\).}
Heuristically, this means that if \(a\) has ``less trace'' than \(b\) then one can cut \(b\) down to \(a\) via multiplication in the algebra.
Changing the point of view this amounts to saying that one may ``move'' \(a\) into \(b\) via elements in the \cstar{}algebra.
Notice that even if \(A \cong C_0(X) \rtimes_r \Gamma\) is a crossed product \cstar{}algebra, strict comparison is blatantly a purely \cstar{}algebraic notion: it talks about general positive elements and arbitrary conjugations, not only positive elements in \(C_0(X)\) and conjugations via \(\gamma \in \Gamma\).
The importance of the relationship between strict comparison and \(\CZ\)-stability cannot be overstated: it is one of the main insights of the Toms-Winter Conjecture. We refer the reader to \cites{winter-2012,kirchberg-rordam,matui-sato,rordam-2004,kerr-2019-comparison} and references therein for a glimpse of these ideas.

On the other hand, as introduced by Winter and then refined by Kerr in \cite{kerr-2019-comparison}, and extensively studied later, one may consider a \emph{dynamical} notion of comparison.
From this perspective (and talking about groupoids, see, \emph{e.g.}\ \cite{ma-comparison-2020}*{Definition 3.4}), given a locally compact, \'etale groupoid \(\CG\), we may ``compare'' two open sets \(\CU, \CV \subseteq \CGo\) of the unit space via open bisections.
Tracial states then correspond to \(\CG\) invariant probability measures, and one says that \(\CG\) has \emph{dynamical comparison} if a behavior similar to the one before holds.
Again, heuristically, this means that if \(\CU\) has ``less measure'' than \(\CV\) then one can rearrange \(\CU\) into \(\CV\) via open bisections.

It is often (but not always) the case that if \(\CG\) has dynamical comparison then \(C^*_r(\CG)\) has strict comparison, see \cite{kerr-2019-comparison}*{Theorems 9.2 and 12.4} and \cite{rordam-2004} for the typical way to show this.
It is however open up to what extent groupoid comparison is automatic (at least in the amenable case).
For instance, Glasner and Liu very recently showed in \cite{glasner-liu-2026}*{Theorem E} that every free action of an amenable group on a finite dimensional, compact metric space has dynamical comparison.
Likewise, \cite{geffen-hua-kranz-2026} proves that a certain class of well-behaved amenable groupoids do have comparison, whereas \cite{boldrini-prasad-2026} proves (non-constructively) there are non-amenable minimal group actions without comparison.
It turns out that the (amenable) groupoid \(\CH\) from \cref{thm-intro} lies on the opposite end of the spectrum: it does not have dynamical comparison at all, and its unit space is infinite dimensional.
We thank Spyridon Petrakos for pointing us to \cite{ma-comparison-2020}*{Theorem B} which, when paired with \cref{thm-intro}, yields the following.

\begin{alphcor}[see \cref{cor:thm-intro}] \label[corollary]{cor-intro}
    The amenable groupoid \(\CH\) in \cref{thm-intro} does not have comparison.
\end{alphcor}

We remark a large part of the proof of \cref{cor-intro} is contained in \cite{rordamsalgebra}.
The groupoid \(\CH\) in \cref{thm-intro} is a model for the corner \(P(D \rtimes_\alpha \ZZ)P\) in \cite{rordamsalgebra}*{Corollary 7.1}, which is known to be ``pathological'' and not have strict comparison.\footnote{\, \cite{rordamsalgebra}*{Corollary 7.1} shows that \(P(D \rtimes_\alpha \ZZ)P\) is infinite, unital and simple, thus has no tracial state. Hence, if it had comparison then every projection would be infinite, for so is the unit \(P\). Nevertheless, \(P(D \rtimes_\alpha \ZZ)P\) has a non-zero finite projection.}
In fact, not only is \(\CH\) a model for \(P(D \rtimes_\alpha \ZZ)P\), it is also constructed using similar methods.
In particular, first we construct a model \(\CG\) for \(D\), then an action \(\widetilde{\alpha} \colon \ZZ \curvearrowright \CG\) dual to R\o rdam's action \(\alpha \colon \ZZ \curvearrowright D\), and then \(\CH\) will be a restriction of \(\CG \rtimes_{\widetilde{\alpha}} \ZZ\).
Notice, moreover, that this groupoid \(\CH\) admits no invariant probability measure, but it is \emph{not} purely infinite (in the groupoid sense).

\vspace{2mm}

The paper is organized as follows.
\cref{sec:preliminaries} contains all the preliminaries needed for the rest of the paper. Of particular interest is \cref{subsec:construction rordam}, where we go through the main milestones of the construction in \cite{rordamsalgebra} and settle notation needed for the rest of the paper.
In fact, there we see that \(D\), the \cstar{}algebra in the algebra of interest \(D \rtimes_\alpha \ZZ\), appears as a subalgebra of the inductive limit of a certain sequence \(\CM(A) \to \CM(A) \to \cdots\).
\cref{sec:Cartan} then exploits a possible choice for this sequence, and shows that it preserves some Cartan structure along the way.
This yields, see \cref{cor:D Cartan}, that \(D\) has a Cartan sub-algebra.
This sub-algebra is then used in \cref{sec:further structure} to produce a Cartan subalgebra of \(D \rtimes_\alpha \ZZ\) itself, which we do by showing that the action \(\alpha \colon \ZZ \curvearrowright D\) in \cite{rordamsalgebra} preserves our Cartan construction.
It is in \cref{sec:further structure} that \cref{thm-intro} is proven.

\vspace{3mm}

\noindent \textbf{Acknowledgments:} this work was done as part of the master thesis of the first-named author.
We would like to thank Spyridon Petrakos and Bartosz Kwa\'{s}niewski for useful comments on a previous draft.

\section{Preliminaries}\label[section]{sec:preliminaries}
This section recalls the preliminaries needed for the rest of the text.
We briefly go through \emph{\'etale groupoids} and their \emph{reduced} \cstar{}algebras in \cref{subsec:groupoid algebras}; and Renault's correspondence in \cref{subsec:groupoid algebras and Cartans}.
Lastly, \cref{subsec:construction rordam} recalls the main points in the construction of R\o rdam's \cite{rordamsalgebra}.

\subsection{Groupoids and their reduced \texorpdfstring{\cstar{}}{C*-}algebras}\label[subsection]{subsec:groupoid algebras}
We (very) briefly recall the notion of \emph{twisted \'etale} groupoid \((\CG, \Sigma)\), and how to construct a \cstar{}algebra out of such a pair \((\CG, \Sigma)\).
For a more comprehensive introduction we refer the reader to \cite{sims2018etalegroupoidscalgebras} and references therein.

A \emph{groupoid} \(\CG\) is a small category with inverses, \emph{i.e.}\ a set \(\CG\) together with an involution \(\gamma \mapsto \gamma^{-1}\), a distinguished set \(\CG^{(2)} \subseteq \CG \times \CG\) of \emph{composable pairs}, and an associative product \(\CG^{(2)} \to \CG\) sending \((\gamma_1, \gamma_2)\) to \(\gamma_1 \gamma_2\) satisfying the usual conditions: \((\gamma_1, \gamma_1^{-1}) \in \CG^{(2)}\); \((\gamma_1\gamma_2)^{-1} = \gamma_2^{-1} \gamma_1^{-1}\); and \(\gamma_1^{-1}\gamma_1 \gamma_2 = \gamma_2\)  and \(\gamma_1\gamma_2\gamma_2^{-1} = \gamma_1\)  whenever \(\gamma_1, \gamma_2 \in \CG\) and \((\gamma_1, \gamma_2) \in \CG^{(2)}\).
We say \(\CG\) is a \emph{topological} groupoid if it is equipped with a topology for which the involution and product are continuous.
The \emph{unit space} of \(\CG\), denoted by \(\CGo \subseteq \CG\), is the set \(\CGo \coloneqq \{\gamma^{-1} \gamma : \gamma \in \CG\} = \{x : (x,x) \in \CG^{(2)} \text{ and } x = x^2\}\).
The \emph{range} map is the map \(\gamma \mapsto \gamma \gamma^{-1}\), whereas the \emph{source} map is \(\gamma \mapsto \gamma^{-1} \gamma\).
Heuristically, an element \(\gamma \in \CG\) can be seen as \(\gamma = \gamma \gamma^{-1} \cdot \gamma \cdot \gamma^{-1} \gamma\), \emph{i.e.}\ an arrow from \(\gamma^{-1} \gamma\) to \(\gamma \gamma^{-1}\) (note one has to read from right to left).
This is a point of view we will follow in the rest of the text.

An open subset \(u \subseteq \CG\) is an (\emph{open}) \emph{bisection} if both the range and source maps are homeomorphisms from \(u\) onto (open) subsets of \(\CGo\).
The groupoid is \emph{\'etale} if its topology can be generated by open bisections.
Notice that if \(u, v \subseteq \CG\) are open bisections then so are
\[
    u \cdot v \coloneqq \{\gamma \rho : \gamma \in u, \rho \in v  \text{ and } (\gamma, \rho) \in \CG^{(2)}\} \hspace{2mm} \text{ and } \hspace{2mm} u^* \coloneqq \{\gamma^{-1} : \gamma \in u\}.
\]
\begin{assumption}
    $\CG$ will always be a second-countable, locally compact, Hausdorff and étale groupoid. 
\end{assumption}
Given a unit \(x \in \CGo\), we denote by \(\CG_x \coloneqq \{\gamma : \gamma^{-1} \gamma = x\}\) the arrows starting at \(x\). Similarly, \(\CG^x \coloneqq \{\gamma : \gamma\gamma^{-1} = x\}\) is the set of arrows ending at \(x\).
Lastly, the \emph{isotropy group at \(x\)} is \(\CG_x^x \coloneqq \CG_x \cap \CG^x\).
Notice that all \(\CG_x, \CG^x\) and \(\CG_x^x\) are discrete should \(\CG\) be \'etale.

We are mainly interested in \'etale groupoids that are \emph{principal}, \emph{topologically principal} or \emph{minimal}.

\begin{definition}\label[definition]{def:principal}
  A groupoid \(\CG\) is \emph{principal} if $\CG^x_x=\{x\}$ for all units $x\in\CGo$. Likewise, \(\CG\) is \emph{topologically principal} if the set of \(x \in \CGo\) with trivial isotropy is dense in \(\CGo\).
  Furthermore, \(\CG\) is \emph{minimal} if the orbit of every unit \(x \in \CGo\) is dense, \emph{i.e.}\ \({\rm Orb}_\CG(x) \coloneqq \{\gamma \gamma^{-1} : \gamma \in G_x\}\) is dense in \(\CGo\).
\end{definition} 

Clearly, a principal groupoid is also topologically principal, but the converse is false.
For technical reasons, and with Renault's correspondence in mind, we need to \emph{twist} \(\CG\). For this we follow the exposition in, for instance, \cite{Renault2008CartanSI}.
A \emph{twist} over \(\CG\) is a (non-\'etale) locally compact, Hausdorff groupoid \(\Sigma\) together with continuous groupoid morphisms \(\iota \colon \CGo \times \TT \to \Sigma\) and \(\pi \colon \Sigma \to \CG\) such that:
\begin{enumerate}
    \item both \(\iota\) and \(\pi\) restrict to homeomorphisms of the unit spaces;
    \item \(\iota\) is injective;
    \item \(\Sigma\) is a locally trivial \(\CG\)-bundle, \emph{i.e.}\  for all \(\gamma \in \CG\) there is an open bisection \(u \subseteq \CG\) containing \(\gamma\) and a continuous section \(s \colon u \to \Sigma\) such that the map
    \[
        u \times \TT \to \pi^{-1}(u), \;\; (\rho, z) \to \iota(\rho\rho^{-1}, z) s(\rho)
    \]
    is a homeomorphism;
    \item \(\iota(\sigma \sigma^{-1}, z) \sigma = \sigma \iota(\sigma^{-1} \sigma, z)\) for all \(z \in \TT\) and \(\sigma \in \Sigma\);
    \item and \(\pi^{-1}(\CGo) = \iota(\CGo \times \TT)\).
\end{enumerate}
Due to the fourth condition above note that we may define \(\sigma z \coloneqq \sigma \iota(\sigma^{-1} \sigma, z)\), in which case we have that \(\sigma z = z \sigma\). We will tacitly use this notation in the sequel.

\begin{remark}
    We often refer to a twist as a short exact sequence \(\CGo \times \TT \to \Sigma \to \CG\), or even a pair \((\CG, \Sigma)\).
\end{remark}

Given an \'etale (locally compact, second countable, Hausdorff) groupoid \(\CG\), we can construct its \emph{reduced} groupoid \cstar{}algebra. This is well known, and the following exposition is based on \cite{sims2018etalegroupoidscalgebras}. First, we define the convolution algebra $C_c(\CG)$ as the set
\begin{equation*}
  C_c(\CG)\coloneqq\{f:\CG\rightarrow\CC:f\text{ continuous and compactly supported}\},
\end{equation*}
with multiplication and involution given by \((f*g)(\gamma)\coloneqq\sum_{\alpha\beta=\gamma}f(\alpha)g(\beta)\) and \(f^*(\gamma)\coloneqq\overline{f(\gamma^{-1})}\).
Notice that this multiplication is well-defined for étale groupoids, for the sets $\CG^{y}$ and $\CG_{x}$ are both discrete.
Hence their intersections with the compact supports of $f$ and $g$ are finite, and thus the multiplication $(f*g)(\gamma)$ is defined by a finite sum.
For each unit $x\in\CGo$, we define the regular representation $\pi_x \colon C_c(\CG)\rightarrow B(l^2(\CG_x))$ of $C_c(\CG)$ associated to $x$ by
\begin{equation*}
  \pi_x(f)\delta_\gamma \coloneqq \sum_{\alpha\in\CG_{r(\gamma)}}f(\alpha)\delta_{\alpha\gamma},
\end{equation*}
where \(\{\delta_\gamma\}_{\gamma \in \CG_x} \subseteq l^2(\CG_x)\) is the usual orthonormal basis.
The \emph{reduced} groupoid \cstar{}algebra $C_r^*(\CG)$ is (by definition) the completion of the image
\(
  (\bigoplus_{x\in\CGo}\pi_x)(C_c(\CG))\subseteq \bigoplus_{x\in\CGo}B(l^2(\CG_x)).
\)
We refer to \cite{sims2018etalegroupoidscalgebras}*{Corollary~3.3.4} for a proof that $\bigoplus_{x\in\CGo}\pi_x$ is injective on \(C_c(\CG)\).

If \(\CGo \times \TT \to \Sigma \to \CG\) is a twist over \(\CG\) then a similar construction may be carried out to yield \(C_r^*(\CG, \Sigma)\).
Since this will play little role in the sequel we only refer the reader to \cite{sims2018etalegroupoidscalgebras}*{Theorem 5.1.11}, and only mention that the Hilbert spaces \(l^2(\CG_x)\) have to be exchanged with the spaces \(L^2(\CG_x, \Sigma_x)\) of square-integrable \(\TT\)-equivariant functions on \(\Sigma_x\).
Likewise, \(C^*_r(\CG, \Sigma)\) is a completion of
\[
    C_c(\CG, \Sigma) \coloneqq \{f \in C_c(\Sigma) : f(z \sigma) = z f(\sigma) \text{ for all } z \in \TT \text{ and } \sigma \in \Sigma\}, 
\]
\emph{i.e.}\ the algebra of continuous, compactly supported, \(\TT\)-equivariant functions on \(\Sigma\).

\begin{remark}
    A twist \(\CGo \times \TT \to \Sigma_{\rm trv} \to \CG\) is \emph{trivial} if it is \(\Sigma_{\rm trv} \cong \CG \times \TT\) and the maps \(\CGo \times \TT \to \Sigma_{\rm trv}\) and \(\Sigma_{\rm trv} \to \CG\) are the canonical ones. 
    In such case both constructions coincide: \(C_r^*(\CG) \cong C_r^*(\CG, \Sigma_{\rm trv})\).
    As later shown, the groupoid model we will construct for R\o rdam's algebra has a trivial twist. 
\end{remark}

\subsection{Cartan subalgebras and Renault's correspondence}\label[subsection]{subsec:groupoid algebras and Cartans}  
We here briefly recall the correspondence between twisted groupoid \cstar{}algebras and Cartan subalgebras of \cstar{}algebras, based on \cite{Renault2008CartanSI} (and partially on \cite{sims2018etalegroupoidscalgebras}).

Suppose $A$ is a \cstar{}algebra and $B \subseteq A$ is a $C^*$-subalgebra.
An element $n\in A$ is a \emph{normalizer} of $B$ if $nBn^*, n^*Bn \subseteq B$.
The collection of all normalizers of $B$ in $A$ is denoted by $N(B\subseteq A)$.
We say that $B$ is \emph{regular} in $A$ if $A$ is generated (as a \cstar{}algebra) by $N(B\subseteq A)$, which means $A = \overline{\text{span}}\,(N(B\subseteq A))$ (for $N(B\subseteq A)$ is multiplicatively closed).
A \emph{conditional expectation} is a completely positive, contractive projection $E \colon A \rightarrow B$ that is the identity on \(B\) and has the property that $E(bab')=bE(a)b'$ for all $a\in A$ and $b,b'\in B$.
It is \emph{faithful} if $E(a^*a)=0$ implies $a=0$ for all $a\in A$. Using these notions, we can recall the notion of \emph{Cartan sub-algebra}, appearing in \cite{Renault2008CartanSI}, and partially based on \cites{kumjian-diagonals-1986,feldman-moore-1997-i1}.
  
\begin{definition}[{see \cite{Renault2008CartanSI}*{Definition~5.1}}]\label[definition]{def:Cartan subalgebra}
  A $C^*$-subalgebra $B$ of a \cstar{}algebra $A$ is a \emph{Cartan subalgebra} if it is is maximally abelian, regular and there exists a faithful conditional expectation $E \colon A\rightarrow B$.
\end{definition}
  
Initially, the definition included the condition that $B$ needs to contain an approximate unit for $A$, but it was later found that this extra condition was redundant, see \cite{Pitts_2023}. If $B$ is a Cartan subalgebra of $A$, the pair $(A,B)$ is called a \emph{Cartan pair}. Notice that for $A$ unital, a Cartan subalgebra $B\subseteq A$ is necessarily unital.
Prior to Renault introducing Cartan subalgebras, Kumjian defined \emph{$C^*$-diagonals} in \cite{kumjian-diagonals-1986}. These are Cartan subalgebras that also have the \enquote{unique extension property}. 
\begin{definition}
  Let $B\subseteq A$ be an inclusion of \cstar{}algebras. $B$ has the \emph{unique extension property relative to $A$} if every pure state on $B$ has a unique extension to a pure state on $A$.
  Furthermore, \(B\) is a \emph{\cstar{}diagonal of \(A\)} if it is a Cartan subalgebra with the unique extension property.
\end{definition}

If $B$ is a \cstar{}diagonal of $A$, we say that $(A,B)$ is a \emph{diagonal pair}.
Notice that if \(E \colon A \to B\) is the conditional expectation of the pair \((A, B)\) and \(\varphi \colon B \to \CC\) is a pure state of \(B\) then \(\varphi \circ E \colon A \to \CC\) is a state of \(A\) extending \(\varphi\), so it is the unique such one.
In \cite{kumjian-diagonals-1986}, Kumjian proves a correspondence between diagonal pairs and \emph{principal} groupoids.
Similarly, in \cite{Renault2008CartanSI}, Renault proves a correspondence between Cartan pairs and \emph{topologically principal} groupoids.
The following is the combination of these.

\begin{theorem}[{see \cite{kumjian-diagonals-1986}*{Theorems~2.9 and~3.1} and \cite{Renault2008CartanSI}*{Theorem 5.9}}] \label[theorem]{theorem:correspondence Renault}
  There is a one-to-one correspondence between twists over second-countable, locally compact, Hausdorff and étale topologically principal groupoids and Cartan pairs $(A,B)$ with $A$ separable via
  \[\big(\CGo\times\TT\rightarrow \Sigma\rightarrow\CG\big)\leftrightarrow\big(C^*_r(\CG,\Sigma),C_0(\CGo)\big).\]
  Moreover, \(B\) is a \cstar{}diagonal in \(A\) if and only if \(\CG\) is principal.
\end{theorem}

\begin{remark}
The proof of \cref{theorem:correspondence Renault} actually explicitly constructs the pair \((\CG, \Sigma)\) out of the pair \((A, B)\).
Typically, these are called the \emph{Weyl groupoid} and \emph{twist} of the pair.
In \cref{subsec:trivial twist} we will need this explicit description, so we will omit the construction here.
\end{remark}

\subsection{Construction of R\o rdam's algebra}\label[subsection]{subsec:construction rordam}
We show how R\o rdam constructed the \cstar{}algebra $D\rtimes_\alpha \ZZ$, following \cite{rordamsalgebra}.
We will maintain the same notation as in \cite{rordamsalgebra}, as we will later need to make certain choices in the construction.

Let $Z \coloneqq \prod_{j = 1}^\infty S^2$ be the infinite product of 2-spheres. Likewise, let $A \coloneqq C(Z)\otimes\CK=C(Z,\CK)$, with $\mathcal{K}$ the compact operators on the infinite dimensional, separable Hilbert space \(\ell^2(\NN)\).
We will always see the elements of $A$ as continuous functions $Z \to \CK$.
From this point of view, elements $f \in \CM(A)$ can be identified with bounded $^*$-strongly continuous functions $f \colon Z\rightarrow \CB(\ell^2(\NN))$. 
  
R\o rdam defines an injective $^*$-homomorphism $\overline{\varphi} \colon \mathcal{M}(A)\rightarrow \mathcal{M}(A)$ as follows. First take a partition $\{\Lambda_r\}_{r=0}^\infty$ of $\NN$ such that $\Lambda_0=\{1\}$ and the other $\Lambda_r$ are infinite. For $r\geq 1$, choose an injective map $\gamma_r \colon \ZZ\times\Lambda_{r-1}\rightarrow\Lambda_r$. Using these maps, define an injective map $\nu \colon \ZZ\times\NN\rightarrow\NN$ by
\begin{equation*}
    \nu(j,t)\coloneqq\gamma_r(j,t) \;\; \text{ for } t\in\Lambda_{r-1} \text{ and } j\in \ZZ.
\end{equation*}
Choose points $c_{j,i}\in S^2$ for all $j,i\in\NN$ with $j\geq i$ that satisfy
\begin{equation*}
  \overline{\{(c_{j,1},c_{j,2},...,c_{j,n}):j\geq n\}}=S^2\times S^2\times \cdots \times S^2
\end{equation*}
for every $n\in \NN$, which can be done since $S^2$ is separable. Define the sets 
\begin{equation*}
  I_j\coloneqq\{\nu(j,1),\nu(j,2), \dots, \nu(j,j)\}
\end{equation*}
for all $j\in \NN$. Using these choices, we define $^*$-homomorphisms $\varphi_j \colon A\rightarrow A$ for all $j$. For $j\leq 0$, set
\begin{equation}\label{eq:phi_j for j<0}
  \varphi_j(f)(x)\coloneqq f(x_{\nu(j,1)},x_{\nu(j,2)},x_{\nu(j,3)}, \dots),
\end{equation}
where $f\in A=C(Z,\CK)$ and $x=(x_1,x_2, \dots)\in Z = \prod_{j = 1}^\infty S^2$.
Whenever $j>0$ the $^*$-homomorphisms are defined in a different way.
For this, let $\tau \colon \CK\otimes\CK\rightarrow\CK$ be any \Star{}isomorphism, and let $p\in C(S^2,M_2(\CC))$ be the Bott projection.
For $n\in \NN$ and $I=\{n_1, \dots,n_k\}$ a finite subset of $\NN$, let $p_n$ and $p_I$ be the projections in $A$ defined by
\begin{align}
  p_n(x)&\coloneqq p(x_n),\label{eq:p_n}\\
  p_I(x)&\coloneqq p(x_{n_1})\otimes p(x_{n_2})\otimes \cdots \otimes p(x_{n_k}) =p_{n_1}(x)\otimes p_{n_2}(x)\otimes \cdots \otimes p_{n_k}(x) \label{eq:p_I}
\end{align}
for all $x=(x_1,x_2, \dots)\in Z = \prod_{j = 1}^\infty S^2$. (Notice here we are tacitly identifying \(M_2, M_2 \tensor M_2, \dots\) with sub-algebras of \(\CK\).)
The $^*$-homomorphism $\varphi_j$ for $j>0$ is now defined by
\begin{equation}\label{eq:phi_j for j>0}
  \varphi_j(f)(x)\coloneqq\tau\big(f(c_{j,1}, \dots,c_{j,j},x_{\nu(j,j+1)},x_{\nu(j,j+2)},\dots)\otimes p_{I_j}(x)\big).
\end{equation}

It is well known that there is some sequence of isometries $\{S_j\}_{j=-\infty}^\infty \subseteq \CM(A)$ with $\sum_{j=-\infty}^\infty S_jS_j^*=1$ (the sum being in the strict topology).
Nevertheless, with future computations in mind let us choose a particular one: take a partition $\{C_j\}_{j=-\infty}^\infty$ of $\NN$ such that $|C_j|=\infty$ for all $j$.
Let $\sigma_j \colon \NN\rightarrow C_j$ be a bijection for all $j$. We may let $S_j(x) \colon \ell^2(\NN)\rightarrow \ell^2(\NN)$ be $S_j(x)\delta_k \coloneqq \delta_{\sigma_j(k)}$ for all $k$ and all $x\in Z$. The map $S_j \colon Z\rightarrow \CB(\ell^2(\NN))$ is constant, so we will denote $S_j(x)$ by $s_j$ for all $x\in Z$. Using these isometries, and as stated in \cite{rordamsalgebra}*{Eq.~(5.5)}, the assignment
\begin{equation*}
  \psi(f)\coloneqq\sum_{j=-\infty}^\infty S_j\varphi_j(f)S_j^*
\end{equation*}
is a $^*$-homomorphism.
In fact, it is ``unital'' in the following way.

\begin{lemma}[{see \cite{rordamsalgebra}*{Lemma~5.1}}]\label[lemma]{lemma:F_j}
  Let $\{e_n\}_{n=1}^\infty$ be an increasing approximate unit for $A$. Then $\{\psi(e_n)\}_{n=1}^\infty$ converges strictly to a projection $F\in \CM(A)$, and $F$ is equivalent to the identity $1$ in $\CM(A)$.
\end{lemma}
\begin{proof}[Sketch of proof]
    For the sake of later reference we sketch the proof.
    For each \(j \leq 0\) one may readily show that \(\varphi_j(e_n) = e_n\), and thus \(\varphi_j(e_n) \to 1 \eqqcolon F_j\) strictly.
    On the other hand, if \(j \geq 1\) then \(\varphi_j(e_n)\) converges strictly to a projection \(F_j(x) \coloneqq \overline{\tau}(1 \tensor p_{I_j}(x))\), where \(x \in Z\).
    Here \(\overline{\tau} \coloneqq \CB(\ell^2(\NN) \tensor \ell^2(\NN)) \to \CB(\ell^2(\NN))\) is a strongly continuous extension the previously chosen isomorphism \(\tau \colon \CK \tensor \CK \to \CK\).
    One may then show that \(\psi(e_n) \to F \coloneqq \sum_{j = -\infty}^{\infty} S_j F_j S_j^*\), where again the sum is in the strict topology.
\end{proof}

\Cref{lemma:F_j} exactly states that there is some isometry $T \in \CM(A)$ with $TT^*=F$.
This allows to define a $^*$-homomorphism $\varphi \colon A\rightarrow\CM(A)$ via
\begin{equation} \label{eq:varphi-sec-two}
  \varphi(f)\coloneqq T^*\psi(f)T=\sum_{j=-\infty}^\infty T^*S_j\varphi_j(f)S_j^*T,
\end{equation}
see \cite{rordamsalgebra}*{Eq.~(5.6)}.
Notice that $\varphi$ maps an approximate unit $\{e_n\}_{n=1}^\infty$ for $A$ into a sequence in $\CM(A)$ that converges strictly to the identity in $\CM(A)$, which follows from \Cref{lemma:F_j} and the choice of $T$.
Now it follows from \cite{lance1995hilbert}*{Proposition~2.5} that this $^*$-homomorphism can be extended to a unital $^*$-homomorphism $\overline{\varphi} \colon \CM(A)\rightarrow\CM(A)$.
Explicitly, we may let
\begin{equation}\label{eq:phi bar}
  \overline{\varphi}(g)\coloneqq\lim_{n\rightarrow\infty}\varphi(e_ng) \;\; \text{ for } g\in \CM(A)
\end{equation}
for the approximate unit $\{e_n\}_{n=1}^\infty\subseteq A$ with $e_n(x)\coloneqq P_n=\sum_{i=1}^n e_{ii}$ for all $x\in Z$. Notice that any choice of approximate unit will yield the same map \(\overline{\varphi}\).
In \cite{rordamsalgebra}*{Proposition~5.2}, R\o rdam proves that \(A \cap \overline{\varphi}(\CM(A)) = 0\), so in particular $\overline{\varphi}$ is injective.
Iterating this procedure leads to a sequence 
\begin{equation} \label{eq:sequence-cm-varphi}
    \CM(A)\xrightarrow{\overline{\varphi}}\CM(A)\xrightarrow{\overline{\varphi}}\CM(A)\xrightarrow{\overline{\varphi}} \cdots
\end{equation}
of unital \cstar{}algebras with unital, injective connecting maps. Let \(B\) be the (neither separable nor exact) inductive limit of this sequence.
Likewise, let $\mu_{m,n}\coloneqq\overline{\varphi}^{(m-n)} \colon \CM(A)\rightarrow\CM(A)$ for $m \geq n$ be the connecting maps from the $n$-th copy of $\CM(A)$ to the $m$-th copy of $\CM(A)$.
From the universal property of $B$, there are inductive limit maps $\mu_{\infty,n} \colon \CM(A)\rightarrow B$.
Since $\overline{\varphi}$ is injective, these inductive limit maps are also injective.

Note that the map $\overline{\varphi}$ extends to an automorphism $\alpha$ on $B$ via $\alpha(\mu_{\infty,n}(x))\coloneqq\mu_{\infty,n}(\overline{\varphi}(x))$ for $x\in \CM(A)$ and $n\in\NN$. On the dense subset $\bigcup_{n=0}^\infty \mu_{\infty,n}(\CM(A))$ of $B$, the inverse of $\alpha$ is given by $\alpha^{-1}(\mu_{\infty,n}(x)) \coloneqq \mu_{\infty,n+1}(x)$. 
In particular, \(\alpha\) induces an action \(\ZZ \curvearrowright B\).
  
Now put $A_0\coloneqq\mu_{\infty,0}(A)\subseteq B$ and $A_n\coloneqq\alpha^n(A_0)\subseteq B$ for all $n\in\ZZ$ and define
\begin{equation*}
  D_n\coloneqq C^*(A_{-n},A_{-n+1}, \dots ,A_0, \dots ,A_{n-1},A_n) \hspace{3mm}\text{ and }\hspace{3mm}D\coloneqq\overline{\bigcup_{n=1}^\infty D_n}.
\end{equation*}
It can be proved that all $D_n$ are type I \cstar{}algebras, see \cite{rordamsalgebra}*{Lemma~6.6}, so $D$ is an inductive limit of type I \cstar{}algebras. Observe the automorphism \(\alpha \colon B \to B\) leaves \(D\) invariant, and thus the following is well posed.

\begin{definition}\label[definition]{def:rordams algebra}
  \emph{R\o rdam's algebra} is the crossed product $D\rtimes_\alpha\ZZ$.
\end{definition}

\begin{theorem}[{cf.\ \cite{rordamsalgebra}*{Theorem~6.10}}] \label[theorem]{theorem:rordams algebra}
  $D\rtimes_\alpha\ZZ$ is nuclear, separable, simple and satisfies the UCT.
    Moreover, it contains an infinite and a finite, non-zero projection.
\end{theorem}

We briefly delve into the ``moreover'' statement above, for it allows to find a unital \cstar{}algebra that is neither \emph{purely infinite}, nor \emph{stably finite}, see \cite{rordamsalgebra}*{Corollary~7.1}.
Recall that a projection $p \in A$ is \textit{infinite} if it is equivalent to a proper subprojection of itself, and \textit{finite} otherwise.
A unital \cstar{}algebra $A$ is called \emph{finite} (resp.\ \emph{infinite}) if $1_A$ is finite (resp.\ infinite).
There are multiple equivalent definitions of ``purely infinite'' \cstar{}algebras.
For simple, unital \cstar{}algebras, we can define it using \enquote{hereditary} subalgebras, see \cite{blackadar-book-2010}*{Proposition~V.2.3.3}. A subalgebra $B$ of a \cstar{}algebra $A$ is called \textit{hereditary} if $0\leq a\leq b$ for $a\in A$, $b\in B$ implies $a\in B$.

\begin{definition}[{see \cite{blackadar-book-2010}*{Proposition~V.2.3.3}}]\label[definition]{def:purely infinite}
  A simple, unital \cstar{}algebra $A\neq \CC$ is \emph{purely infinite} if and only if every nonzero hereditary subalgebra of $A$ contains an infinite projection.
\end{definition}

In \cite{blackadar-book-2010}*{Proposition~V.2.2.24} Blackadar proves that every nonzero projection in a purely infinite \cstar{}algebra is infinite.
On the other end of the spectrum, we say \(A\) is \emph{stably finite} if $M_n(A)$ is finite for all $n \in \NN$, see \cite{blackadar-book-2010}*{Definition~V.2.1.5}.
Observe that while every stably finite \cstar{}algebra is finite, the converse is not true in general, see \cite{blackadar-book-2010}*{Example~V.2.1.6}.

In our setting, we may define the projections $Q\coloneqq\mu_{\infty,0}(p_1)$, where $p_1$ is the projection defined in \eqref{eq:p_n}, and \mbox{$P\coloneqq\mu_{\infty,0}(g)$} with $g$ a constant one-dimensional projection in $A$.
In \cite{rordamsalgebra}*{Theorem~6.10} R\o rdam proves the following.
  
\begin{proposition}\label[proposition]{prop:projections rordam}
  $Q$ is a nonzero finite projection, and $P$ is an infinite projection in $D\rtimes_\alpha\ZZ$.
\end{proposition}
  
\begin{corollary}[{see \cite{rordamsalgebra}*{Corollary~7.1}}]\label[corollary]{cor:corner rordam}
  The corner $P(D\rtimes_\alpha\ZZ)P$ is a nuclear, unital, separable, simple \cstar{}algebra satisfying the UCT. It is infinite, but contains a nonzero finite projection.
\end{corollary}

This implies that the corner $P(D\rtimes_\alpha\ZZ)P$ is neither stably finite nor purely infinite. Since the corner is a nuclear, unital, separable, simple \cstar{}algebra satisfying the UCT, it provides a counterexample to \Cref{question: intro 2}.
The main goal of the paper is the construction of a groupoid model for \(P (D \rtimes_{\alpha} \ZZ) P\): this is, in fact, the groupoid \(\CH\) appearing in \cref{thm-intro}.
  
To end this section, we state some technical properties that are needed in the next sections. All of these are taken out from \cite{rordamsalgebra}*{Sections 5 and 6}.
\begin{proposition}\label[proposition]{prop:properties rordam}
  With notations as above, the following properties hold.
  \begin{enumerate}
    \item \label{prop:properties rordam:i}   $\overline{\varphi} \colon \CM(A) \to \CM(A)$ is injective.
    \item \label{prop:properties rordam:ii}  $\mu_{\infty,n} \colon \CM(A) \to B$ is injective for all $n$.
    \item \label{prop:properties rordam:iii} $\overline{\varphi}\big(\CM(A)\big) \cap A=\{0\}$.
    \item \label{prop:properties rordam:iv}  $A_mA_n\subseteq A_{\min\{m,n\}}$.
    \item \label{prop:properties rordam:v}   $D_n=A_{-n}+A_{-n+1}+\dots+A_n$, and each element $d\in D_n$ can be written as a sum $d=d_{-n}+d_{-n+1}+\dots+d_n$ with $d_i\in A_i$ in a unique way.
  \end{enumerate}
\end{proposition}

\section{A Cartan subalgebra of \texorpdfstring{\(D\)}{D}}\label[section]{sec:Cartan}

In this section we prove that the \cstar{}algebra $D$ constructed in \Cref{subsec:construction rordam} has a Cartan subalgebra $\widetilde{D}$.
We will later use this \(\widetilde{D}\) in order to construct a groupoid model for R\o rdam's algebra $D\rtimes_\alpha\ZZ$.
To construct the pair $\widetilde{D} \subseteq D$ we follow exactly the same construction as for $D$, but starting from a Cartan subalgebra $\widetilde{A} \subseteq A$.
Indeed, the \cstar{}algebra $A=C(Z,\CK)$ has a groupoid model, for it has a canonical Cartan subalgebra $\widetilde{A} \coloneqq C(Z,c_0(\NN))\subseteq C(Z,\CK) = A$.

\begin{convention}
    We will use the same notations as in \cref{subsec:construction rordam} throughout the rest of the paper.
\end{convention}

\subsection{Construction of the connecting map}\label[subsection]{subsec:construction phi}

Before constructing the subalgebra $\widetilde{D} \subseteq D$ we give an explicit construction for the map $\overline{\varphi} \colon \CM(A) \to \CM(A)$ by finding a description of the isometry $T \in \CM(A)$ that was used in the definition of $\varphi \colon A \to \CM(A)$ in \eqref{eq:varphi-sec-two}.
For this, we have to construct an isometry $T\in \CM(A)$ with $TT^*=F$, see \cref{lemma:F_j}.
Recall that the projection $p\in C(S^2,M_2(\CC))$ used in the construction of the maps $\varphi_j$ for $j>0$ in \eqref{eq:phi_j for j>0} is chosen to be the Bott projection $p\in C(S^2,M_2(\CC))$, defined by
\begin{equation*}
  p(z)\coloneqq \frac{1}{1+|z|^2}\begin{pmatrix}
    \abs{z}^2&z\\
    \overline{z}&1
  \end{pmatrix} \;\; \text{ for } z\in \CC.
\end{equation*}
(Here we are identifying \(S^2\) with \(\CC \sqcup \{\infty\}\), \emph{i.e.}\ the one point compactification of \(\CC\), and noting that if \(\abs{z} \to \infty\) and we define \(p(\infty) \coloneqq e_{11}\) then the resulting \(p\) really is continuous.)
This gives 
\begin{equation} \label{eq:p-rank-one}
  p(z)\begin{pmatrix}
    0\\
    1
  \end{pmatrix}=\frac{1}{1+|z|^2}\begin{pmatrix}
  z\\
  1
  \end{pmatrix}\hspace{5mm}\text{ and }\hspace{5mm}
  p(z)\begin{pmatrix}
    1\\
    0
  \end{pmatrix}=\overline{z}\cdot p(z)\begin{pmatrix}
  0\\
  1
  \end{pmatrix},
\end{equation}
which implies that \({\rm im} (p(z)) = \{(\lambda z, \lambda) : \lambda \in \CC\}\) for all \(z \in \CC \subseteq S^2\).
Furthermore, we may choose the isomorphism $\tau \colon \CK\otimes\CK\rightarrow \CK$ by fixing a bijection \mbox{$\phi \colon \NN\times \NN\rightarrow\NN$} and defining $\tau$ on the elementary tensors by
\begin{equation}\label{eq:tau}
  \tau(e_{ij}\otimes e_{kl})\coloneqq e_{\phi(i,k)\phi(j,l)}.
\end{equation}
Equivalently, we may induce a unitary \(U \colon \ell^2(\NN) \tensor \ell^2(\NN) \to \ell^2(\NN)\) by \(\delta_{p} \tensor \delta_q \mapsto \delta_{\phi(p,q)}\), and then let \(\tau(k \tensor k') \coloneqq U^* (k \tensor k') U\).

\begin{convention} \label[convention]{conv:infty-in-s-two}
    From now given \(x \in S^2\) we may see it as an element \(x \in \CC\), with the adequate understanding that \(x\) may be \(\infty\), in which case \(p(x) = e_{11} \in M_2\).
\end{convention}

For $x\in Z$ and $j>0$, we have $F_j(x) = \overline{\tau}(1\otimes p_{I_j}(x))$ as defined in (the proof of) \Cref{lemma:F_j}. Suppose $I_j=\{n_1,n_2,\dots,n_j\}$, now
\begin{align*}
  p_{I_j}(x)&=p_{n_1}(x)\otimes\dots\otimes p_{n_j}(x)\\
  &=\frac{1}{1+|x_{n_1}|^2}\begin{pmatrix}
    |x_{n_1}|^2&x_{n_1}\\ \overline{x}_{n_1}&1
  \end{pmatrix}\otimes \cdots \otimes \frac{1}{1+|x_{n_j}|^2}\begin{pmatrix}
    |x_{n_j}|^2&x_{n_j}\\ \overline{x}_{n_j}&1
  \end{pmatrix}.
\end{align*}
If \(\{e_{ij}\}_{i,j \in \{1,2\}} \subseteq M_2\) are the usual matrix units, the above expression is the same as
\begin{align*}
  p_{I_j}(x)=\bigotimes_{i=1}^j\frac{1}{1+|x_{n_i}|^2}\big(|x_{n_i}|^2e&_{11}+x_{n_i}e_{12}+\overline{x}_{n_i}e_{21}+e_{22}\big).
\end{align*}
Recall we have tacitly chosen particular embeddings \(M_2 \subseteq M_3 \subseteq \cdots \subseteq \CK\).
Under the usual identification \(M_2 \tensor M_2 \cong M_4\) we have that \(e_{i,j} \tensor e_{p,q} = e_{2(i-1)+p,2(j-1)+q}\), each side understood in their own matrix algebra (and hence the right hand side being a \(4\) by \(4\) matrix).
In our setting, however, we need to take \(j\) tensor products of elements in \(M_2\), see \eqref{eq:p_I}.
In the following, by convention, we say that empty products equal \(1\).
We are grateful to Spyridon Petrakos for finding a mistake in a previous version of the text, which the following solves.

\begin{lemma}
    We have that
    \begin{equation} \label{eq:p-I-in deltas}
        p_{I_j}(x) = \frac{1}{\prod_{i = 1}^j 1 + \abs{x_{n_i}}^2} \sum_{\fa, \fb \in \{1,2\}^j} x_{\fa} \overline{x}_\fb \, e_{\varsigma(\fa) \varsigma(\fb)},
    \end{equation}
    where \(x_\fa \coloneqq \prod_{i : \fa_i = 1} x_{n_i}\), similarly for \(\overline{x}_\fb \coloneqq \prod_{i : \fb_i = 1} \overline{x}_{n_i}\), and \(\varsigma(\fa) \coloneqq 1+ \sum_{i = 1}^j (\fa_i - 1) 2^{j-i}\).
\end{lemma}
\begin{proof}
    First note that the normalization item in \eqref{eq:p-I-in deltas} is just a number, and it is the product of the analogue numbers in the definition of \(p_{I_j}(x)\) above. So we will ignore these quantities for the sake of readability, and let \(q_{I_j}(x)\) be the same as in \eqref{eq:p-I-in deltas} but without these.
    We proceed by induction on \(j\).
    If \(j = 1\) then we only have \(1\) tensor product, and the formula is indeed correct.
    In general, suppose the formula \eqref{eq:p-I-in deltas} holds up for \(j - 1\) and \(1\); that is, 
    \[
        q_{I_{j-1}}(x) = \sum_{\fa, \fb \in \{1,2\}^{j-1}} x_{\fa} \overline{x}_\fb \, e_{\varsigma_{j-1}(\fa) \varsigma_{j-1}(\fb)} \; \text{ and } \; q_{n_{j}}(x) = \sum_{\fc, \fd \in \{1,2\}} x_{\fc} \overline{x}_\fd \, e_{\varsigma_{1}(\fc) \varsigma_{1}(\fd)},
    \]
    where \(\varsigma_{j-1}\) and \(\varsigma_1\) are the same as in the statement but summing up to \(j-1\) and \(1\) respectively.
    Hence notice that \(\varsigma_1(\fc) = \fc_1\), and similarly for \(\fd\).
    Consider the tensor product \(q_{I_{j-1}}(x) \tensor q_{n_j}(x)\).
    The claim is then a tedious computation:
    \begin{align*}
        q_{I_{j-1}}(x) \tensor q_{n_j}(x) & = \left(\sum_{\fa, \fb \in \{1,2\}^{j-1}} x_{\fa} \overline{x}_\fb \, e_{\varsigma_{j-1}(\fa) \varsigma_{j-1}(\fb)}\right) \tensor \left(\sum_{\fc, \fd \in \{1,2\}} x_\fc \overline{x}_\fd e_{\varsigma_{1}(\fc) \varsigma_1(\fd)}\right) \allowdisplaybreaks \\
        & = \sum_{\fa, \fb \in \{1,2\}^{j-1}} \sum_{\fc, \fd \in \{1,2\}} x_\fa x_\fc \cdot \overline{x}_{\fb} \overline{x}_\fd \cdot \left(e_{\varsigma_{j-1}(\fa) \varsigma_{j-1}(\fb)} \tensor e_{\varsigma_{1}(\fc) \varsigma_{1}(\fd)}\right) \allowdisplaybreaks \\
        & = \sum_{\fa, \fb \in \{1,2\}^j} x_\fa \overline{x}_\fb \cdot e_{2(\varsigma_{j-1}(\fa|_{1, \dots, j-1}) - 1) + \varsigma_{1}(\fa_j), 2(\varsigma_{j-1}(\fb|_{1, \dots, j-1}) - 1) + \varsigma_{1}(\fb_j)} \allowdisplaybreaks \\
        & = \sum_{\fa, \fb \in \{1,2\}^j} x_\fa \overline{x}_\fb e_{\varsigma_j(\fa) \varsigma_j(\fb)},
    \end{align*}
    where the previous-to-last equality follows from the fact that \(e_{i,j} \tensor e_{p,q} = e_{2(i-1)+p,2(j-1)+q}\) under the isomorphism \(M_2 \tensor M_2 \cong M_4\) alluded to before.
    Likewise, in the last equality we used that
    \[
        \varsigma_j(\fa) = 1 + \sum_{i = 1}^j \left(\fa_i - 1\right) 2^{j-i} = 2\left(1 + \sum_{i = 1}^{j-1} (\fa_i - 1) 2^{j-1 - i} - 1\right) + \fa_j= 2(\varsigma_{j-1}(\fa|_{1, \dots, j-1}) - 1) + \varsigma_{1}(\fa_j),
    \]
    and similarly for \(\varsigma_j(\fb)\).
\end{proof}

Putting \(c(x_{I_j}) \coloneqq 1/\prod_{i = 1}^j 1 + \abs{x_{n_i}}^2\), it follows from \eqref{eq:tau} and \eqref{eq:p-I-in deltas} that $F_j(x) $ is given by
\begin{align}\label{eq:F_j(x)}
  F_j(x) \stackrel{{\rm (def)}}{=} \sum_{l \in \NN} \tau(e_{ll} \otimes p_{I_j}(x)) \stackrel{{\rm \eqref{eq:tau} + \eqref{eq:p-I-in deltas}}}{=} c(x_{I_j}) \sum_{l \in \NN} \sum_{\fa, \fb \in \{1,2\}^j} x_\fa \overline{x}_\fb \, e_{\phi(l, \varsigma(\fa)) \phi(l,\varsigma(\fb))}.
\end{align}
Note that, should \(x \in Z = \prod_{j = 1}^\infty S^2 \cong \prod_{j = 1}^\infty (\CC \sqcup \{\infty\}) \) have an ``\(\infty\)'' coordinate then technically the expression \(F_j(x)\) above would have to be appropriately substituted  via \cref{conv:infty-in-s-two}.

Coming back to \(p_{I_j}(x)\) in \eqref{eq:p_I} and \eqref{eq:p-I-in deltas}, again notice that it is (by definition) a tensor product of \(j\) rank one projections. By \eqref{eq:p-rank-one} these are, moreover, the projections onto the normalized vectors
\[
   w_{n_i}(x) \coloneqq \frac{1}{\sqrt{1 + \abs{x_{n_i}}^2}} (x_{n_i}, 1) \in \CC^2.
\]
Hence the projection \(p_{I_j}(x)\) itself is also rank one, and it projects onto the vector \(w_{I_j}(x) \coloneqq w_{n_1}(x) \tensor \cdots \tensor w_{n_j}(x) \in \CC^{2^j}\).
Thus \(1 \tensor p_{I_j}(x) = \sum_{l \in \NN} e_{ll} \tensor p_{I_j}(x)\) projects onto the linear span of \(\{\delta_l \tensor w_{I_j}(x)\}_{l \in \NN}\).
Similarly to \eqref{eq:p-I-in deltas}, and using the same language, notice we may express \(w_{I_j}(x)\) more succinctly as
\begin{equation} \label{eq:w-I-j-vector}
    w_{I_j}\left(x\right) = c(x_{I_j})^{1/2} \sum_{\fa \in \{1,2\}^j} x_\fa \delta_{\varsigma(\fa)}.
\end{equation}
The following is now essentially obvious, recall \eqref{eq:tau}.

\begin{proposition}\label[proposition]{prop:V_j isometry}
  Let $F_j \colon Z\rightarrow \CB(\ell^2(\NN))$ be as in \eqref{eq:F_j(x)}. An orthonormal basis for $\operatorname{range}\big(F_j(x)\big)$ is
  \begin{equation*}
    \left\{v_{(l,j)}(x) \coloneqq c(x_{I_j})^{1/2} \sum_{\fa \in \{1,2\}^j} x_\fa \delta_{\phi(l,\varsigma(\fa))}\right\}_{l \in \NN}.
  \end{equation*}
  Thus, the map \(V_j \colon Z \to \CB(\ell^2(\NN))\) given by \(V_j(x)\delta_l \coloneqq v_{(l,j)}(x)\) satisfies that \(V_j(x)\) is an isometry with \(V_j(x) V_j(x)^* = F_j(x)\) for all $x\in Z$. 
\end{proposition}
\begin{proof}
  By definition \(F_j(x) = \overline{\tau}(1 \tensor p_{I_j}(x))\), which is a projection, for \(\tau\) is a \Star{}homomorphism, and hence so is \(\overline{\tau}\).
  Moreover, going back to the discussion immediately after \eqref{eq:tau} we can see that \(\tau\) is spatially implemented by the unitary operator \(U \colon \ell^2(\NN) \tensor \ell^2(\NN) \to \ell^2(\NN)\), where \(\delta_p \tensor \delta_q \mapsto \delta_{\phi(p,q)}\).
  Thus \(U\) takes the image of \(1 \tensor p_{I_j}(x)\) to the image of \(F_j(x)\), so the claim follows from \eqref{eq:w-I-j-vector} and the formula for \(U\).
  Notice that orthogonality of the collection \(\{v_{(l,j)}(x)\}_{l \in \NN}\) follows from that of \(\{\delta_l \tensor w_{I_j}(x)\}_{l \in \NN}\).

  The ``thus'' statement follows readily.
  Indeed, \(\{\delta_l\}_{l \in \NN}\) forms a basis of \(\ell^2(\NN)\). Similarly, \(\{v_{(l,j)}(x)\}_{l \in \NN}\) forms a basis of the range of \(F_j(x)\), and hence sending one basis to the other is indeed an isometry with the desired properties.
\end{proof}

As clarification, we briefly remark that the collection \(\{v_{(l,j)}(x)\}_{l \in \NN}\) is indeed orthogonal, but the larger collection \(\{v_{(l,j)}(x)\}_{l, j \in \NN}\) may not be. This will pose no problem though, as we will always work on a \(j\)-by-\(j\) basis.
The first instance of this is the following.

\begin{lemma}\label[lemma]{lemma:V_j continuous}
  The map $V_j \colon Z\rightarrow \CB(\ell^2(\NN))$ in \cref{prop:V_j isometry} is $^*$-strongly continuous. Thus \(V_j \in \CM(A)\).
\end{lemma}
\begin{proof}
  Suppose $x_k\rightarrow x$ in $Z$ and fix some $\zeta= \sum_{l \in \NN} \zeta_l \delta_l \in \ell^2(\NN)$, we want to prove that
  \begin{equation*}
    \norm{\big(V_j(x_k)-V_j(x)\big)\zeta}\rightarrow 0 \;\; \text{ and } \;\; \norm{\big(V_j(x_k)^*-V_j(x)^*\big)\zeta}\rightarrow 0 \;\; \text{ when } k \to \infty.
  \end{equation*}
  Recall the definition of \(w_{I_j}(x)\) from \eqref{eq:w-I-j-vector}, and notice that
  \begin{align*}
    \norm{\left(V_{j}\left(x_k\right) - V_j\left(x\right) \right) \zeta} & = \norm{\sum_{l \in \NN} \zeta_l \left(v_{(l,j)}\left(x_k\right) - v_{(l,j)}\left(x\right) \right)} \stackrel{{\rm \eqref{eq:tau}}}{=} \norm{U^* \sum_{l \in \NN} \zeta_l \left(v_{(l,j)}\left(x_k\right) - v_{(l,j)}\left(x\right) \right)} \\
    & = \norm{\sum_{l \in \NN} \zeta_l \left(\delta_l \tensor \left(w_{I_j}(x_k) - w_{I_j}(x)\right)\right)} \leq \norm{\zeta} \cdot \norm{w_{I_j}\left(x_k\right) - w_{I_j}\left(x\right)},
  \end{align*}
  which goes to zero since $x_{k}\to x$ and the sums defining \(w_{I_j}(\variable)\) contain only finitely many elements, at most \(2^j\) to be exact, see \eqref{eq:w-I-j-vector}.
  Using that $V_j(x_k)$ is an isometry for any $x_k\in Z$, which we know from \Cref{prop:V_j isometry}, we find
  \begin{align*}
    \norm{\big(V_j(x_k)^*-V_j(x)^*\big)\zeta}&=\norm{\big(V_j(x_k)V_j(x_k)^*-V_j(x_k)V_j(x)^*\big)\zeta}\\
    &\leq\norm{\big(V_j(x_k)V_j(x_k)^*-V_j(x)V_j(x)^*\big)\zeta} +\norm{\big(V_j(x)-V_j(x_k)\big)\big(V_j(x)^*\zeta\big)}\\
    &=\norm{\big(F_j(x_k)-F_j(x)\big)\zeta}+\norm{\big(V_j(x)-V_j(x_k)\big)\big(V_j(x)^*\zeta\big)},
  \end{align*}
  where the second term goes to zero by the previous argument, and the first term goes to zero by strong continuity of $\overline{\tau}$ and the fact that \(F_j(x) = \overline{\tau}(1 \tensor p_{I_j}(x))\).
\end{proof}

We may now use the maps $V_j$ to construct an isometry $T$ as in \eqref{eq:varphi-sec-two}.
For this, consider
\begin{equation}\label{eq:T_k}
  T_k\coloneqq\sum_{j=-k}^0S_jS_j^*+\sum_{j=1}^k S_jV_jS_j^*,
\end{equation}
which belongs to $\CM(A)$ for all $k \in \NN$ since $S_j\in \CM(A)$ and $V_j\in \CM(A)$ for all $j > 0$ by \cref{lemma:V_j continuous}.

\begin{lemma}\label[lemma]{lemma:strict convergence T_k}
  If $T_k$ is as in \eqref{eq:T_k}, then the sequence $\{T_k\}_{k\in\NN}$ is strictly convergent.
\end{lemma}
\begin{proof}
  Take $a\in A$. Since the multiplier algebra $\CM(A)$ is complete in the strict topology, it suffices to prove that $\norm{(T_l-T_k)a}\rightarrow 0$ and $\norm{a(T_l-T_k)}\rightarrow 0$ for $l,k\rightarrow\infty$. For this, first observe that for all $k$, we have
  \begin{equation*}
    T_k^*T_k=\sum_{j=-k}^0 S_jS_j^*+\sum_{j=1}^k S_jV_j^*V_jS_j^*=\sum_{j=-k}^k S_jS_j^*\leq 1,
  \end{equation*}
  which follows from \Cref{prop:V_j isometry} and the fact that $\sum_{j=-\infty}^\infty S_jS_j^* =1$. This implies $\norm{T_k^*T_k}\leq 1$ and thus $\norm{T_k}\leq 1$ for all $k$. For all $l,k$ with $l>k$, we have
  \begin{equation*}
    \norm{(T_l-T_k)a}=\norm{\Big(\sum_{j=-l}^{-k-1}S_jS^*_j+\sum_{j=k+1}^l S_jV_jS_j^*\Big)a}
    =\norm{T_l\Big(\sum_{j=-l}^{-k-1}S_jS^*_j+\sum_{j=k+1}^l S_jS_j^*\Big)a},
  \end{equation*}
  where we use the fact that $S_j^*S_i=0$ when $j\neq i$. Furthermore, we have
  \begin{align*}
    \norm{T_l\Big(\sum_{j=-l}^{-k-1}S_jS^*_j+\sum_{j=k+1}^l S_jS_j^*\Big)a} &\leq \norm{T_l}\cdot\norm{\Big(\sum_{j=-l}^{-k-1}S_jS^*_j+\sum_{j=k+1}^l S_jS_j^*\Big)a}\\
    &\leq\norm{\Big(\sum_{j=-l}^{-k-1}S_jS^*_j+\sum_{j=k+1}^l S_jS_j^*\Big)a}.
  \end{align*}
  Since $\sum_{j=-\infty}^\infty S_jS_j^*=1$ and the sum is strictly convergent, the latter expression goes to zero. This implies $\norm{(T_l-T_k)a}\rightarrow 0$.
    An analogous argument shows $\norm{a(T_l-T_k)}\rightarrow 0$.
\end{proof}

\begin{corollary}\label[corollary]{cor:T well-defined}
  The operator \(T\coloneqq\sum_{j=-\infty}^0S_jS_j^*+\sum_{j=1}^\infty S_jV_jS_j^*\) belongs to $\CM(A)$.
    Moreover, \(T\) is an isometry whose range projection is \(F\).
\end{corollary}
\begin{proof}
    The fact that the sum defining \(T\) is strictly convergent follows from \cref{lemma:strict convergence T_k} and the discussion around \eqref{eq:T_k}.
    Using \Cref{prop:V_j isometry}, we find for all $x\in Z$
    \begin{equation*}
     T^*T(x)=\sum_{j=-\infty}^0 s_js_j^*+\sum_{j=1}^\infty s_jV_j(x)^*V_j(x)s_j^*=\sum_{j=-\infty}^\infty s_js_j^*=1.
    \end{equation*}
    Likewise, \(TT^*(x) = \sum_{j=-\infty}^0 s_js_j^* + \sum_{j=1}^\infty s_jV_j(x)V_j(x)^*s_j^* =  F(x)\), as desired.
\end{proof}

We now use this isometry $T$ in our definition of $\varphi \colon A \to \CM(A)$ (cf.\ \eqref{eq:varphi-sec-two}):
\begin{align}\label{eq:phi}
    \varphi \colon A & \to \CM\left(A\right), \\
  f & \mapsto \sum_{j=-\infty}^\infty T^*S_j\varphi_j(f)S_j^*T= \sum_{j=-\infty}^0 S_j\varphi_j(f)S_j^*+\sum_{j=1}^\infty S_jV_j^*\varphi_j(f)V_jS_j^*, \nonumber
\end{align}
where $V_j^* \colon Z\rightarrow \CB(\ell^2(\NN))$ is defined by $V_j^*(x) \coloneqq V_j(x)^*$ for all $x\in Z$.
As in \eqref{eq:phi bar}, we may extend \(\varphi\) to a unital \(^*\)-homormophism \(\overline{\varphi} \colon \CM(A) \to \CM(A)\) in a unique manner.
Indeed, for elements $f \in \CM(A)$, we can describe $\overline{\varphi}$ by taking the limit of the sequence $\{\varphi(e_nf)\}_{n\in\NN}$ as in \eqref{eq:phi bar}.
In the rest of the paper, we will use this description of $\overline{\varphi}$ to determine some properties of the \cstar{}algebra $D$ in \cref{def:rordams algebra}.

\vspace{2mm}

We start by using \eqref{eq:phi} to prove that $\overline{\varphi} \colon \CM(A) \to \CM(A)$ preserves certain structure: it preserves \emph{diagonal elements} and \emph{normalizers}.
For this, recall there is an obvious faithful representation $\CM(A) \to \ell^\infty(Z, \CB(\ell^2(\NN))) \coloneqq \prod_{x \in Z} \CB(\ell^2(\NN))$, simply by forgetting the topology of \(Z\).
Indeed, we are tacitly always assuming that elements \(f \in \CM(A)\) are functions \(f \colon Z \to \CB(\ell^2(\NN))\) (with certain continuity conditions).
By \emph{diagonal elements} in $\CM(A)$ we mean elements of the subalgebra $\diagelms \coloneqq \ell^\infty(Z, \ell^\infty(\NN)) \cap \CM(A)$, where we see \(\ell^\infty(\NN) \subseteq \CB(\ell^2(\NN))\) in the usual manner.
We observe the following.

\begin{lemma} \label[lemma]{lemma:diagelms}
    \(\diagelms\) is precisely the set of functions \(f \colon Z \to \ell^\infty(\NN)\) that are \(^*\)-strongly continuous, that is, \(f(x_k)(n) \to f(x)(n)\) for all \(n \in \NN\) whenever \(x_k \to x\) in \(Z\).
\end{lemma}
\begin{proof}
    Simply notice that elements \(f \in \CM(A)\) are \(^*\)-strongly continuous functions \(f \colon Z \to \CB(\ell^2(\NN))\).
\end{proof}

\begin{lemma}\label[lemma]{lemma:diagonality preserved}
  If $f\in \diagelms$, then $\overline{\varphi}(f)\in \diagelms$.
\end{lemma}
\begin{proof}
    By definition \(\diagelms = \ell^\infty(Z, \ell^\infty(\NN)) \cap \CM(A)\).
    As \(\overline{\varphi}\) maps \(\CM(A)\) into \(\CM(A)\) it suffices to show that \(\overline{\varphi}(f) \in \ell^\infty(Z, \ell^\infty(\NN))\) for every function \(f \in \ell^\infty(Z, \ell^\infty(\NN)) \cap \CM(A)\).

  We will first assume that $f\in C\big(Z,c_0(\NN)\big)$, so $f\in A$ and $f(x) \in c_0(\NN)$ for all $x\in Z$. Now, by \eqref{eq:phi}, we have that $\overline{\varphi}(f)=\varphi(f)=\sum_{j=-\infty}^0S_j\varphi_j(f)S_j^*+\sum_{j=1}^\infty S_jV_j^*\varphi_j(f)V_jS_j^*$.
    Thus
  \begin{align*}
    \varphi(f)(x)=&\sum_{j=-\infty}^0s_jf(x_{\nu(j,1)},\dots)s^*_j+\sum_{j=1}^\infty s_jV^*_j(x)\tau\big(f(c_{j,1},\dots,c_{j,j},x_{\nu(j,j+1)},\dots)\otimes p_{I_j}(x)\big)V_j(x)s^*_j.
  \end{align*}
  It is clear that all terms of the first sum are still diagonal because of the way the isometries $s_j$ are defined.
  In order to prove that the second sum is still diagonal, it suffices to check that every summand
  \begin{equation*}
    G_j(x)\coloneqq V^*_j(x)\tau\big(f(c_{j,1}, \dots,c_{j,j}, x_{\nu(j,j+1)}, \dots)\otimes p_{I_j}(x)\big)V_j(x)
  \end{equation*}
  is diagonal, where $j>0$ and $x\in Z$.
    Observe
  \begin{align*}
    \tau\big(f(c_{j,1},\dots,c_{j,j},x_{\nu(j,j+1)}, \dots)\otimes p_{I_j}(x)\big)&=\overline{\tau}\big(f(c_{j,1},\dots,c_{j,j},x_{\nu(j,j+1)},\dots)\otimes 1\big)\overline{\tau}\big(1\otimes p_{I_j}(x)\big)\\
    &=\overline{\tau}\big(f(c_{j,1},\dots,c_{j,j},x_{\nu(j,j+1)},\dots)\otimes 1\big)V_j(x)V_j(x)^*,
  \end{align*}
  where we used \Cref{prop:V_j isometry}. This implies \(G_j(x) = V_j(x)^*\overline{\tau}\big(f(c_{j,1}, \dots,c_{j,j},x_{\nu(j,j+1)}, \dots)\otimes 1\big)V_j(x)\), since $V_j(x)^*V_j(x)=1$ by \Cref{prop:V_j isometry} again. Expressing \(f(c_{j,1}, \dots,c_{j,j}, x_{\nu(j,j+1)},\dots)=\sum_k a_k e_{kk}\),
  it follows that
  \(\overline{\tau}\big(f(c_{j,1},\dots,c_{j,j},x_{\nu(j,j+1)},\dots)\otimes 1\big)=\sum_{k,l} a_k e_{\phi(k,l)\phi(k,l)}\), recall \eqref{eq:tau}.
  This implies that
  \begin{align*}
    \overline{\tau}\big(& f(c_{j,1}, \dots,c_{j,j},x_{\nu(j,j+1)},\dots) \otimes 1\big) V_j(x) \delta_h = \sum_{k, l \in \NN} a_k e_{\phi(k,l) \phi(k,l)} v_{(h,j)}(x) \allowdisplaybreaks \\
    & = \sum_{k, l \in \NN} a_k e_{\phi(k,l) \phi(k,l)} \sum_{\fa \in \{1,2\}^j} c(x_{I_j})^{1/2} x_\fa \delta_{\phi(h,\varsigma(\fa))} = \sum_{\fa \in \{1,2\}^j} c(x_{I_j})^{1/2} x_\fa a_h \delta_{\phi(h,\varsigma(\fa))} = a_h v_{(h,j)}(x).
  \end{align*}
  Thus
  \begin{align*}
    \langle\delta_{h'},G_j(x)\delta_h\rangle&=\langle V_j(x)\delta_{h'},\overline{\tau}\big(f(c_{j,1},...,c_{j,j},x_{\nu(j,j+1)},...)\otimes 1\big)V_j(x)\delta_h\rangle =\langle v_{(h',j)}(x),a_h v_{(h,j)}(x)\rangle.
  \end{align*}
    If $h\neq h'$ this inner product is zero (see \cref{prop:V_j isometry}). Whence $G_j(x)$ is diagonal for all $j$ and all $x$, which in turn implies that $\varphi(f)(x)$ is diagonal for all $x$.
    
    For a general $f\in \ell^\infty(Z, \ell^\infty(\NN)) \cap \CM(A)$ we have $\overline{\varphi}(f)(x)=\lim_{n\rightarrow\infty}\varphi(fe_n)(x)$, where $e_n(x)=P_n$ with $P_n$ the projection onto the first $n$ basis vectors $\{\delta_i\}_{i = 1}^n \subseteq \ell^2(\NN)$.
    Each $\varphi(fe_n)(x)$ is diagonal, and hence so is the limit.
\end{proof}

Having proved that \(\overline{\varphi} \colon \CM(A) \to \CM(A)\) preserves diagonal elements we turn our attention to the fact that it commutes with the following canonical conditional expectation \(\overline{E} \colon \CM(A) \to \ell^\infty(Z, \ell^\infty(\NN))\).
\begin{definition} \label[definition]{def:e-bar}
    Let $P \colon \CB(\ell^2(\NN))\rightarrow l^\infty(\NN)$ be the usual (faithful) conditional expectation, \emph{i.e.}\ \(P(t)(k) \coloneqq \innerprod{t \delta_k}{\delta_k}\).
    Denote by $\overline{E} \colon \CM(A)\rightarrow \ell^\infty(Z, l^\infty(\NN))$ the map that sends an element $f\in \CM(A)$ to its diagonal part, so $\overline{E}(f)(x) \coloneqq P(f(x))$ for all $x\in Z$. 
\end{definition}

Note $\overline{E}$ is a faithful conditional expectation, for so is \(P\). 
Furtermore, as posed, \cref{def:e-bar} states no continuity condition on the image \(\overline{E}(f)\).
The following remedies this.

\begin{lemma} \label[lemma]{lemma:image-e-bar}
    The image of \(\overline{E} \colon \CM(A) \to \ell^\infty(Z, \ell^\infty(\NN))\) is exactly \(\diagelms\).
\end{lemma}
\begin{proof}
    Let \(f \in A = C(Z, \CK)\) be given.
    By continuity of \(f\) we have that \(f(x_k) \to f(x)\) (in norm) whenever \(x_k \to x\) (in \(Z\)).
    Since the conditional expectation \(P \colon \CB(\ell^2(\NN)) \to \ell^\infty(\NN)\) is contractive it follows that \(\overline{E}(f)(x_k) \to \overline{E}(f)(x)\), as desired.
    A general \(f \in \CM(A)\) is precisely a \(^*\)-strongly continuous \(f \colon Z \to \CB(\ell^2(\NN))\), and a similar argument as before shows the statement.
\end{proof}

By \cref{lemma:diagelms,lemma:image-e-bar} we henceforth see \(\overline{E}\) as a map \(\overline{E} \colon \CM(A) \to \diagelms\).

\begin{corollary}\label[corollary]{cor:E commutes with phi}
  For all $f\in \CM(A)$, we have $\overline{\varphi}\big(\overline{E}(f)\big)=\overline{E}\big(\overline{\varphi}(f)\big)$. In other words, the following diagram is commutative:
\begin{center}\begin{tikzcd}
  {\CM(A)} & {\CM(A)} \\
  {\diagelms} & {\diagelms}
  \arrow["{\overline{\varphi}}", from=1-1, to=1-2]
  \arrow["\overline{E}", from=1-1, to=2-1]
  \arrow["\overline{E}", from=1-2, to=2-2]
  \arrow["{\overline{\varphi}}", from=2-1, to=2-2].
\end{tikzcd} \end{center}
\end{corollary} 

\begin{proof}
  We may split any given $f\in \CM(A)$ into its diagonal part $\overline{E}(f)$ and its off-diagonal part $f-\overline{E}(f)$.
    As $\overline{\varphi}(f)=\overline{\varphi}\big(\overline{E}(f)\big)+\overline{\varphi}\big(f-\overline{E}(f)\big)$, by \cref{lemma:diagonality preserved} it suffices to show that $\overline{\varphi}(f-\overline{E}(f))$ is off-diagonal.
    We do this by proving that $\overline{\varphi}\big(f-\overline{E}(f)\big)(x)$ cannot contain multiples of $e_{ii}$ for any $i$ and any $x\in Z$.
    We may first assume that $f\in A$, whence $\overline{\varphi}(f - \overline{E}(f)) = \varphi\big(f-\overline{E}(f)\big)$. Using \eqref{eq:phi}, we find 
  \begin{equation*}
    \varphi(f-\overline{E}(f))=\sum_{j=-\infty}^0S_j\varphi_j(f-\overline{E}(f))S_j^*+\sum_{j=1}^\infty S_jV_j^*\varphi_j(f-\overline{E}(f))V_jS_j^*.
  \end{equation*}
  Clearly, the terms for $j\leq 0$ will not give any diagonal terms because of the definition of $\varphi_j$ in \eqref{eq:phi_j for j<0} and the way we defined the $S_j$. Take some arbitrary $j>0$. By the definition of the $S_j$, it suffices to look at $G_j\coloneqq V_j^*\varphi_j(f-\overline{E}(f))V_j$.
    Given any $x\in Z$ we have that
  \begin{equation*}
    G_j(x)=V_j(x)^*\overline{\tau}\big((f-\overline{E}(f))(c_{j,1},\dots,c_{j,j},x_{\nu(j,j+1)},\dots)\otimes 1\big)V_j(x),
  \end{equation*}
  which follows from the proof of \Cref{lemma:diagonality preserved}.
    To simplify notation let $y \coloneqq (c_{j,1},\dots,c_{j,j},x_{\nu(j,j+1)},\dots) \in Z$.
    The element $(f-\overline{E}(f))(y)$ is of the form $\sum_{k \neq l}a_{kl}e_{kl}$ for some $\{a_{kl}\}_{k \neq l}\subseteq\CC$, since $f-\overline{E}(f)$ is exactly the off-diagonal part of $f$. This implies
  \begin{align}
    \overline{\tau}\big((f-\overline{E}(f))(y)\otimes 1\big)&=\lim_{n\rightarrow\infty} \tau\big((f-\overline{E}(f))(y)\otimes P_n\big)= \lim_{n\rightarrow\infty}\sum_{k\neq l,m\leq n} a_{kl}e_{\phi(k,m)\phi(l,m)}\notag\\
    &=\sum_m \sum_{k\neq l} a_{kl}e_{\phi(k,m)\phi(l,m)},\label{eq:tau of off-diagonal}
  \end{align}
  where we used the expression of $\tau$ from \eqref{eq:tau}.
    This implies that $\overline{\tau}\big((f-\overline{E}(f))(y)\otimes 1\big)$ will only contain off-diagonal terms since $\phi$ is injective. Using this, we can then see that
  \begin{align*}
    \langle \delta_h ,G_j(x)\delta_h\rangle & = \langle V_j(x)\delta_h ,\overline{\tau}\big((f-\overline{E}(f))(y)\otimes 1\big)V_j(x)\delta_h\rangle \allowdisplaybreaks \\
    & \stackrel{{\rm \eqref{eq:tau of off-diagonal}}}{=} \langle v_{(h,j)}(x), \sum_m \sum_{k \neq l} \sum_{\fa \in \{1,2\}^j} x_\fa a_{kl} e_{\phi(k,m)\phi(l,m)} \delta_{\phi(h, \varsigma(\fa))} \rangle \allowdisplaybreaks \\
    & = \sum_{\fa, \fb \{1,2\}^j} \sum_{k \neq h} x_\fb \overline{x}_\fa \overline{a}_{kh} \cdot \langle \delta_{\phi(h,\varsigma(\fb))}, \delta_{\phi(k,\varsigma(\fa))} \rangle = 0,
  \end{align*}
  for all $h \in \NN$ and all $x\in Z$, where the last equality holds since \(k \neq h\) and \(\phi\) is injective, recall also \cref{prop:V_j isometry} for the notation we used.
    This implies that $G_j(x)$ does not contain multiples of $e_{ii}$ for any $i$ and $x\in Z$, as claimed, whence $\varphi\big(\overline{E}(f)\big)=\overline{E}\big(\varphi(f)\big)$.
    If $f\in \CM(A)$, by \eqref{eq:phi bar} we have 
  \begin{equation*}
    \overline{\varphi}\big(\overline{E}(f)\big)=\lim_{n\rightarrow\infty}\varphi\big(e_n\overline{E}(f)\big)=\lim_{n\rightarrow \infty}\overline{E}\big(\varphi(e_n f)\big)=\overline{E}\big(\lim_{n\rightarrow\infty}\varphi(e_n f)\big)=\overline{E}\big(\overline{\varphi}(f)\big),
  \end{equation*}
  which proves the statement.
\end{proof}
  
\subsection{Construction of the Cartan subalgebra \texorpdfstring{$\widetilde{D}$}{D}}\label[subsection]{subsec:construction Cartan}  
We may now construct a subalgebra $\widetilde{D}$ for the algebra $D$ from \Cref{subsec:construction rordam} by following the Cartan pair
\[
    \widetilde{A} \coloneqq C\left(Z\right) \tensor c_0 \subseteq C\left(Z\right) \tensor \CK \cong C\left(Z, \CK\right) = A
\]
along the sequence \eqref{eq:sequence-cm-varphi}.
Note that \(\widetilde{A}\) is the usual Cartan subalgebra of \(A\), as \(c_0\) is the unique (up to unitary conjugacy) Cartan subalgebra of \(\CK\).
More precisely, see \eqref{eq:sequence-cm-varphi}, we have an inductive system \(\CM(A) \to \CM(A) \to \cdots\), whose connecting maps are \(\overline{\varphi}\), and whose limit is (by definition) \(B\).
Let $\widetilde{A}_0\coloneqq\mu_{\infty,0}(\widetilde{A})$ and $\widetilde{A}_n\coloneqq\alpha^n(\widetilde{A}_0)$ for all $n\in\ZZ$, similar to how we defined $A_0$ and $A_n$ in \Cref{sec:preliminaries}.
Likewise, let
\begin{equation*}
  \widetilde{D}_n\coloneqq C^*(\widetilde{A}_{-n},\widetilde{A}_{-n+1},\dots,\widetilde{A}_0,\dots,\widetilde{A}_{n-1},\widetilde{A}_n) \;\; \text{ and } \;\; \widetilde{D} \coloneqq \overline{\bigcup_{n=1}^\infty \widetilde{D}_n}.
\end{equation*}
By \Cref{lemma:diagonality preserved} the sequence \eqref{eq:sequence-cm-varphi} restricts to $\diagelms$, \emph{i.e.}\ this diagram commutes: 
\begin{equation} \label{eq:big-diagram}
    \begin{tikzcd}
      {\CM(A)} & {\CM(A)} & {\CM(A)} & \cdots & B & D \\
      {\diagelms} & {\diagelms} & {\diagelms} & \cdots & C & \widetilde{D}.
      \arrow["{\overline{\varphi}}", from=1-1, to=1-2]
      \arrow["{\overline{E}}", from=1-1, to=2-1]
      \arrow["{\overline{\varphi}}", from=1-2, to=1-3]
      \arrow["{\overline{E}}", from=1-2, to=2-2]
      \arrow["{\overline{\varphi}}", from=1-3, to=1-4]
      \arrow["{\overline{E}}", from=1-3, to=2-3]
      \arrow["{\overline{\varphi}}", from=1-4, to=1-5]
      \arrow["{\overline{\varphi}}", from=2-1, to=2-2]
      \arrow["{\overline{\varphi}}", from=2-2, to=2-3]
      \arrow["{\overline{\varphi}}", from=2-3, to=2-4]
      \arrow["{\overline{\varphi}}", from=2-4, to=2-5]
      \arrow[phantom, "\supseteq", from=1-6, to=1-5]
      \arrow[phantom, "\supseteq", from=2-6, to=2-5]
      \arrow[phantom, "\rotatebox{90}{$\subseteq$}", from=2-6, to=1-6]
    \end{tikzcd}
\end{equation}
Here $C$ is the inductive limit of the lower sequence and the diagram is commutative by \Cref{cor:E commutes with phi}. The algebra $D$ is a subset of the inductive limit $B$. Since $\widetilde{A}_n\subseteq C$ for all $n$, we have $\widetilde{D}\subseteq C$. 

Recall from \Cref{prop:properties rordam} that 
\(D_n=A_{-n}+A_{-n+1}+\dots+A_n\),
where we can write each element of $D_n$ as a sum of elements in $A_{-n},A_{-n+1},\dots,A_n$ in a unique way. We can prove a similar result for $\widetilde{D}_n$.
  
\begin{lemma}\label[lemma]{lemma:D_n tilde as sum}
  We have \(\widetilde{D}_n=\widetilde{A}_{-n}+\cdots+\widetilde{A}_n,\)
  and each element $d\in \widetilde{D}_n$ can be written as a sum $d=d_{-n}+\cdots+d_n$ with $d_i\in \widetilde{A}_i$ in a unique way.
\end{lemma}
\begin{proof}
  Since $\widetilde{A}\subseteq A$, it follows from \Cref{prop:properties rordam}~\cref{prop:properties rordam:iii} that $\overline{\varphi}(\widetilde{A})\cap \widetilde{A}=\{0\}$. Furthermore, since $\widetilde{A}_n\subseteq A_n$ for all $n$ and the product of diagonal elements is again diagonal, it follows from \Cref{prop:properties rordam}~\cref{prop:properties rordam:iv} that $\widetilde{A}_m\widetilde{A}_n\subseteq \widetilde{A}_{\min\{m,n\}}$. Now the result follows in the same way as \Cref{prop:properties rordam}~\cref{prop:properties rordam:v}.
\end{proof} 
  
\cref{lemma:D_n tilde as sum} will be used several times in the sequel.
As $\alpha$ is a \(^*\)-isomorphism and $\mu_{\infty,0}$ is injective (by \Cref{prop:properties rordam}), one can see that $\alpha^n\circ \mu_{\infty,0}$ is a $^*$-isomorphism between $A$ and $A_n$ carrying $\widetilde{A}$ to $\widetilde{A}_n$ (by definition).
This implies that $\widetilde{A}_n$ is a Cartan subalgebra of $A_n$ for all $n \in \ZZ$.
The rest of this section is dedicated to proving that the Cartan structure is preserved, meaning that $\widetilde{D} \subseteq D$ is still a Cartan pair.
We have to prove that $\widetilde{D}$ is a regular masa and there exists a faithful conditional expectation $E \colon D\rightarrow \widetilde{D}$.

\begin{remark}
    The expectation \(E \colon D \to \widetilde{D}\) will be the restriction of \(\overline{E} \colon \CM(A) \to \diagelms\) to \(D\).
    However, note that the existence of the diagram \eqref{eq:big-diagram} is not enough, for we aim to apply a result from \cite{li-raad-2023} which, in principle, only applies to \emph{separable} algebras.
    On the contrary, \(\CM(A)\) and \(\diagelms\) are most definitely not (norm) separable.
\end{remark}

\begin{theorem}\label[theorem]{theorem:masa}
  $\widetilde{D}_n$ is maximally abelian in $D_n$ for all $n\in\NN$.
\end{theorem}
\begin{proof}
    Take $n\in \NN$ and suppose $a\in D_n$ commutes with all elements in $\widetilde{D}_n$. We want to prove that $a\in \widetilde{D}_n$. We can use \Cref{prop:properties rordam} to write \(a=a_{-n}+a_{-n+1}+\dots+a_n\) for some $\{a_i\}_{i = -n}^n \subseteq A_i$ in a unique way.
    Noting $a_i=\alpha^i(\mu_{\infty,0}(b_i))$ for some $b_i\in A$, one has that
  \begin{align*}
    a&=\mu_{\infty,n}(b_{-n})+\mu_{\infty,n-1}(\overline{\varphi}(b_{-n+1}))+\dots+\mu_{\infty,0}(\overline{\varphi}^{2n}(b_n))\\
    &=\mu_{\infty,n}\big(b_{-n}+\overline{\varphi}(b_{-n+1})+\dots+\overline{\varphi}^{2n}(b_n)\big)\eqqcolon\mu_{\infty,n}(f).
  \end{align*}
  Denote by $p_{ii}$ the element in $A$ with $p_{ii}(x)=e_{ii}$ for all $x\in Z$. Now $a$ will commute with $\mu_{\infty,n}(p_{ii})$ for all $i$, so $\mu_{\infty,n}(p_{ii}f)=\mu_{\infty,n}(fp_{ii})$ for all $i$. By \Cref{prop:properties rordam}, we know that  $\mu_{\infty,n}$ is injective, which implies that $f$ commutes with all $p_{ii}$. This means that $f(x)$ commutes with all $e_{ii}$, so $f(x)$ is diagonal for all $x\in Z$. 
    
  Now we prove that each term $\overline{\varphi}^{n+i}(b_i)$ is actually diagonal. We can split $b_{-n}$ into its diagonal and off-diagonal part, so $b_{-n}=\overline{E}(b_{-n})+\big(b_{-n}-\overline{E}(b_{-n})\big)$. Since $b_{-n}(x)$ is a compact operator for all $x$, $\big(b_{-n}-\overline{E}(b_{-n})\big)(x)$ will also be compact for all $x$. We have $f=b_{-n}+\overline{\varphi}(g)$ for some $g\in \CM(A)$, and $f$ is diagonal, implying that the off-diagonal part of $b_{-n}+\overline{\varphi}(g)$ is zero. This implies 
  \begin{equation*}
    b_{-n}+\overline{\varphi}(g)-\overline{E}\big(b_{-n}+\overline{\varphi}(g)\big)=0,
  \end{equation*} 
  so \(\overline{\varphi}(g)-\overline{E}\big(\overline{\varphi}(g)\big)=-b_{-n}+\overline{E}(b_{-n})\in A\).
  From \Cref{cor:E commutes with phi}, we know that $\overline{\varphi}(g)-\overline{E}\big(\overline{\varphi}(g)\big)=\overline{\varphi}\big(g-\overline{E}(g)\big)$. Likewise, \Cref{prop:properties rordam} yields $\overline{\varphi}(\CM(A))\cap A=\{0\}$, which implies  
  \begin{equation*}
    0=\overline{\varphi}\big(g-\overline{E}(g)\big)=b_{-n}-\overline{E}(b_{-n}).
  \end{equation*}
  Now $b_{-n}$ and $\overline{\varphi}(g)$ are both diagonal, so $\overline{\varphi}(g)=\overline{E}\big(\overline{\varphi}(g)\big)=\overline{\varphi}\big(\overline{E}(g)\big)$. By injectivity of $\overline{\varphi}$, which also follows from \Cref{prop:properties rordam}, we find $g=\overline{E}(g)$, so $g$ is diagonal. We can now use the same reasoning to find that $b_{-n+1}$ is diagonal.
  Iterating this process yields that all $b_i$ belong to $\widetilde{A}$, which gives
  \begin{equation*}
    a=\alpha^{-n}(\mu_{\infty,0}(b_{-n}))+...+\alpha^n(\mu_{\infty,0}(b_n))
  \end{equation*}
  with $b_i\in \widetilde{A}$ for all $i$, and whence $a\in \widetilde{D}_n$.
\end{proof}
  
We turn to the proof that $\widetilde{D}$ is regular in $D$.
For this we first need some preliminary lemmas, the first of which shows that it suffices to prove that the images in $A_i$ of normalizers of $\widetilde{A}$ are normalizers of $\widetilde{D}$. 
  
\begin{lemma}\label[lemma]{lemma:regularity1}
  $D$ is generated as a \cstar{}algebra by the elements of the form $\alpha^n\big(\mu_{\infty,0}(b)\big)$ for $b\in N(\widetilde{A}\subseteq A)$.
\end{lemma}
\begin{proof}
  Take any $d\in D$ and $\varepsilon>0$.
  By the inductive limit construction we can find $d_n\in D_n$ for some $n$ such that $\norm{d-d_n}<\varepsilon/2$. By \Cref{prop:properties rordam} we have 
  \begin{equation*}
    d_n=\mu_{\infty,n}(a_{-n})+\mu_{\infty,n-1}(a_{-n+1})+\dots+\mu_{\infty,0}(\overline{\varphi}^n(a_n))
  \end{equation*}
  for some $\{a_i\}_{i = -n}^n \subseteq A$.
  Since $\widetilde{A} \subseteq A$ is regular there are normalizers $b_{i,j}$ such that 
  \begin{equation*}
    \norm{a_i-(b_{i,1}+\dots+b_{i,k_i})}<\frac{\varepsilon}{2(2n+1)}
  \end{equation*}
  for all $i=-n,-n+1,\dots,n$.
  Since the $\mu_{\infty,n}$ and $\overline{\varphi}$ are injective $^*$-homomorphisms by \Cref{prop:properties rordam} and thus isometric, we have
  \begin{align*}
    \norm{d_n-\big(\mu_{\infty,n}(b_{-n,1}+\dots+b_{-n,k_{-n}})+\dots+\mu_{\infty,0}(\overline{\varphi}^n(b_{n,1}+\dots+b_{n,k_n}))\big)} <\varepsilon/2.
  \end{align*}
  Thus, if we let
  \(
    c\coloneqq \mu_{\infty,n}(b_{-n,1}+\dots+b_{-n,k_{-n}})+\dots+\mu_{\infty,0}(\overline{\varphi}^n(b_{n,1}+\dots+b_{n,k_n})),
  \)
  then $\norm{d-c}<\varepsilon$ and $c$ belongs to the \cstar{}algebra generated by the elements of the form $\alpha^n(\mu_{\infty,0}(b))$ for $b\in N(\widetilde{A}\subseteq A)$.
\end{proof}
  
Note that for $b\in N(\widetilde{A}\subseteq A)$, we have $b(x)\in N(c_0(\NN)\subseteq\CK)$ for all $x\in Z$. This implies that each $b(x)$ is of the form $\sum a_{kl}e_{kl}$ where for each $k$ in the sum, there is a unique non-zero $l$ and vice versa. The following lemma shows that $\overline{\varphi}^k(b)$ will still have a similar form for a normalizer $b$.
In the following (and within its proof) if we write $\sum a_{kl}e_{kl}$ then we mean that every term \(a_{kl}\) is non-zero.
  
\begin{lemma}\label[lemma]{lemma:regularity2}
  If $b\in N(\widetilde{A}\subseteq A)$, then for each $i\in\NN$ and $x\in Z$, $\overline{\varphi}^i(b)(x)$ is of the form $\sum a_{kl}e_{kl}$ where $a_{kl}\neq0$ implies $a_{kj}=0=a_{hl}$ for all $j\neq l$ and $h\neq k$.
\end{lemma}
\begin{proof}
  It is enough to prove that $\overline{\varphi}(b)(x)$ is of this form whenever $b\in \CM(A)$ with $b(x)=\sum a_{kl}e_{kl}$ where each $k$ corresponds to a unique $l$ and vice versa for all $x\in Z$. We first assume that $b\in A$. In this case, we know from the proof of \Cref{lemma:diagonality preserved} that
  \begin{equation*}
    \overline{\varphi}(b)(x)=\sum_{j=-\infty}^0 s_jb(x_{\nu(j,1)},\dots)s_j^*+\sum_{j=1}^\infty s_jV_j(x)^*\overline{\tau}\big(b(c_{j,1},\dots,c_{j,j},x_{\nu(j,j+1)},\dots)\otimes1\big)V_j(x)s_j^*
  \end{equation*}
  for all $x\in Z$. Take $x\in Z$ and define $y_j\in Z$ for all $j>0$ by $y_j \coloneqq (c_{j,1},\dots,c_{j,j},x_{\nu(j,j+1)},\dots)$. Take $j>0$ arbitrary.
  By the definition of the $s_j$ it is enough to look at each term separately.
  For $j\leq0$, the terms are clearly of the right form.
  This implies that we only have to prove that \mbox{$V_j(x)^*\overline{\tau}\big(b(y_j)\otimes 1\big)V_j(x)$} is of the right form for all $j>0$.
  By the definition of the map $\overline{\tau}$ in \eqref{eq:tau} one sees that the factor \mbox{$\overline{\tau}\big(b(y_j)\otimes 1\big)$} is of the form $\sum_m \sum_{k \neq l} a_{kl}e_{\phi(k,m)\phi(l,m)}$ where each $k$ corresponds to a unique $l$ and vice versa. Injectivity of $\phi$ implies that each $\phi(k,m)$ corresponds to a unique $\phi(l,m)$ and vice versa. Now take $h\in \NN$, we have
  \begin{align*}
    \overline{\tau}\big(b(y_j) & \otimes 1\big) V_j(x) \delta_h = \sum_{m \in \NN} \sum_{k \neq l} a_{kl}e_{\phi(k,m)\phi(l,m)}v_{(h,j)}(x) \\
    & = \sum_{m \in \NN} \sum_{k \neq l} a_{kl} e_{\phi(k,m)\phi(l,m)} \sum_{\fa \in \{1,2\}^j} c(x_{I_j})^{1/2} x_\fa \delta_{\phi(h, \varsigma(\fa))} \\
    & = \sum_{\fa \in \{1,2\}^j} a_{k_0 h} c(x_{I_j})^{1/2} x_\fa \delta_{\phi(k_0,\varsigma(\fa))} = a_{k_0h} v_{(k_0,j)}(x)
  \end{align*}
  where \(k_0 \in \NN\) is the necessarily unique index such that \(a_{k_0h} \neq 0\) (if such an index exists).
  Applying $V_j(x)^*$ now gives 
  \begin{equation*}
    V_j(x)^*\overline{\tau}\big(b(y_j)\otimes 1\big)V_j(x)\delta_h = V_j(x)^* a_{k_0h} v_{(k_0,j)}(x) = a_{k_0h} \delta_{k_0}
  \end{equation*}
  so each $\delta_h$ gets mapped to zero or to a multiple of $\delta_{k_0}$ for a unique index \(k_0\) depending only on \(h\) (and on \(j\)). This precisely means that if $V_j(x)^*\overline{\tau}\big(b(y_j)\otimes 1\big)V_j(x)=\sum b_{st}e_{st}$, and $b_{st}\neq 0$ implies $b_{it}=0$ for all $i\neq s$. To prove that $b_{st}\neq 0$ also implies $b_{si}=0$ for all $i\neq t$, it suffices to prove that the adjoint sends each $\delta_h$ to zero or to a multiple of some $\delta_t$. One may readily show this by noting
  \begin{equation*}
    \big(V_j^*(x)\overline{\tau}\big(b(y_j)\otimes 1\big)V_j(x)\big)^*=V_j^*(x)\overline{\tau}\big(b(y_j)\otimes 1\big)^*V_j(x)
  \end{equation*}
  and following a very similar argument.

  In the general case where $b \notin \CM(A)$, as usual we have \(\overline{\varphi}(b)=\lim_{n\rightarrow\infty}\varphi(e_nb)\),
  where $e_n(x)=P_n$ is the projection onto the first $n$ basis vectors of $\ell^2(\NN)$. For each $x\in Z$, $(e_nb)(x)$ is of the form $\sum_{k\leq n}a_{kl}e_{kl}$, where each $k$ corresponds to a unique $l$ and vice versa. So $\varphi(e_nb)$ will be of the right form for all $n$. This will still be the case when we take the limit for $n\rightarrow \infty$, finishing the proof.
\end{proof}

\begin{theorem}\label[theorem]{theorem:regularity}
  $\widetilde{D}$ is regular in $D$.
\end{theorem}
\begin{proof}
  By \Cref{lemma:regularity1} it suffices to prove that $\alpha^n\big(\mu_{\infty,0}(b)\big)$ belongs to $N(\widetilde{D}\subseteq D)$ for all integers $n$ whenever $b$ belongs to $N(\widetilde{A}\subseteq A)$. Take $b\in N(\widetilde{A}\subseteq A)$ and $f\in \widetilde{A}$, we want to prove that
  \begin{equation}\label{eq:regularity1}
    \alpha^n\big(\mu_{\infty,0}(b)\big)\alpha^m\big(\mu_{\infty,0}(f)\big)\alpha^n\big(\mu_{\infty,0}(b^*)\big)\in \widetilde{D}
  \end{equation}
  for all $m,n$. Using the fact that $\widetilde{D}_n=\widetilde{A}_{-n}+\widetilde{A}_{-n+1}+\dots+\widetilde{A}_n$ from \Cref{lemma:D_n tilde as sum}, this implies that
  \(
    \alpha^n\big(\mu_{\infty,0}(b)\big)c\alpha^n\big(\mu_{\infty,0}(b^*)\big)\in \widetilde{D}
  \)
  for all $c\in \widetilde{D}$, since $c$ is a limit of elements in $\widetilde{D}_n$.

  In order to prove \eqref{eq:regularity1} it suffices to prove that 
  \begin{equation}\label{eq:regularity2}
    \mu_{\infty,n}(\overline{\varphi}^k(b))\mu_{\infty,m}(\overline{\varphi}^l(f))\mu_{\infty,n}(\overline{\varphi}^k(b^*))\in \widetilde{D}
  \end{equation}
  for all $n,m,l,k\geq 0$. The expression \eqref{eq:regularity2} can be rewritten as
  \begin{align*}
    \mu_{\infty,n+m}\big(\mu_{n+m,n}(\overline{\varphi}^k(b))\mu_{m+n,m}(\overline{\varphi}^l(f))\mu_{m+n,n}(\overline{\varphi}^k(b^*))\big) = 
    \mu_{\infty,n+m}\big(\overline{\varphi}^{k+m}(b)\overline{\varphi}^{l+n}(f)\overline{\varphi}^{k+m}(b^*)\big).
  \end{align*}
  Whence it suffices to look at $\overline{\varphi}^{k+m}(b)\overline{\varphi}^{l+n}(f)\overline{\varphi}^{k+m}(b^*)$. First suppose $k+m<l+n$, now 
  \begin{equation*}
    \overline{\varphi}^{k+m}(b)\overline{\varphi}^{l+n}(f)\overline{\varphi}^{k+m}(b^*)=\overline{\varphi}^{k+m}\big(b\overline{\varphi}^{l+n-(k+m)}(f)b^*\big)\eqqcolon\overline{\varphi}^{k+m}(g).
  \end{equation*}
  Since the compact operators form a two-sided ideal and $b(x)$ is compact for all $x\in Z$, $g(x)$ will be compact for all $x\in Z$. Furthermore, $\overline{\varphi}^{l+n-(k+m)}(f)(x)$ is diagonal for all $x$, so $g(x)$ will still be diagonal because of the form of $b(x)$. This implies that $g\in C(Z,c_0(\NN))$, so
  \begin{align*}
    \mu_{\infty,n}(\overline{\varphi}^k(b))\mu_{\infty,m}(\overline{\varphi}^l(f))\mu_{\infty,n}(\overline{\varphi}^k(b^*))&=\mu_{\infty,n+m}\big(\overline{\varphi}^{k+m}(g)\big)=\mu_{\infty,n}(\overline{\varphi}^k(g))\in \widetilde{D}.
  \end{align*}
  On the other hand, suppose $k+m=l+n$. In this case:
  \begin{align*}
    \mu_{\infty,n}(\overline{\varphi}^k(b))\mu_{\infty,m}(\overline{\varphi}^l(f))\mu_{\infty,n}(\overline{\varphi}^k(b^*))&=\mu_{\infty,n}(\overline{\varphi}^k(bfb^*))\in \widetilde{D}
  \end{align*}
  since $bfb^*\in C(Z,c_0(\NN))$. Lastly, whenever $k+m>l+n$ one may follow a similar argument and finish the proof.
\end{proof}

Lastly, we now construct a faithful conditional expectation $E_n \colon D_n \rightarrow \widetilde{D}_n$. For this, we start from the faithful conditional expectation $E' \colon A\rightarrow \widetilde{A}$ given by restricting the map $\overline{E} \colon \CM(A)\rightarrow \diagelms$ from \Cref{def:e-bar}, so $E'$ sends $f\in A$ to its diagonal part.
  
\begin{theorem}\label[theorem]{theorem:faithful conditional expectation}
  There exists a faithful conditional expectation $E_n \colon D_n\rightarrow \widetilde{D}_n$ for each $n\in\NN$.
\end{theorem}
\begin{proof}
  
    Take $n\in\NN$. Using that $\alpha^n\circ\mu_{\infty,0}$ is a $^*$-isomorphism between $A$ and $A_n$, we can use the faithful conditional expectation $E' \colon A\rightarrow C(Z,c_0(\NN))$ to construct a faithful conditional expectation $E'_n \colon A_n\rightarrow \widetilde{A}_n$ by decreeing
  \begin{equation*}
    E'_n(\alpha^n(\mu_{\infty,0}(f)) \coloneqq \alpha^n(\mu_{\infty,0}(E'(f)))
  \end{equation*}
    for all $f\in A$. For sums, and since $D_n=A_{-n}+\dots+A_n$ by \Cref{prop:properties rordam}, one can define
  \begin{equation*}
    E_n \colon D_n\rightarrow \widetilde{D}_n:a=a_{-n}+\dots +a_n\mapsto E'_{-n}(a_{-n})+\dots+E'_n(a_n).
  \end{equation*}
  This is well-defined, for every element in $D_n$ can be written as a sum of elements of $A_{-n},\dots,A_n$ in a unique way.
  Since each $E'_i$ is a projection onto $\widetilde{A}_i$, $E_n$ is a projection onto $\widetilde{D}_n$. We will now prove that $E_n$ is contractive. Take an element $a\in D_n$, it is of the form
  \begin{equation*}
    a=\mu_{\infty,n}\big(b_{-n}+\overline{\varphi}(b_{-n+1})+\dots+\overline{\varphi}^{2n}(b_n)\big)\eqqcolon\mu_{\infty,n}(f).
  \end{equation*}
  Using \Cref{cor:E commutes with phi}, we find
  \begin{align}\label{eq:E_n(a)}
    E_n(a)&=\mu_{\infty,n}\big(E'(b_{-n})+\overline{\varphi}(E'(b_{-n+1}))+\dots+\overline{\varphi}^{2n}(E'(b_n))\big)\notag\\
    &=\mu_{\infty,n}\big(\overline{E}\big(b_{-n}+\overline{\varphi}(b_{-n+1})+\dots+\overline{\varphi}^{2n}(b_n)\big)\big)=\mu_{\infty,n}\big(\overline{E}(f)\big).
  \end{align}
  Taking the diagonal part cannot increase the norm, so $E_n$ is contractive since $\mu_{\infty,n}$ is an isometry. The map $E_n$ is a contractive projection, so it is a conditional expectation by \cite{Brown-Ozawa:Approximations}*{Theorem~1.5.10}. In order to show it is faithful, take $a\in D_n$ with $E_n(a^*a)=0$. We have $a=\mu_{\infty,n}(f)$ for some $f\in \CM(A)$, so $a^*a=\mu_{\infty,n}(f^*f)$. Analogous to \eqref{eq:E_n(a)} one finds \(0=E_n(a^*a)=\mu_{\infty,n}\big(\overline{E}(f^*f)\big)\).
  Since $\mu_{\infty,n}$ is injective by \Cref{prop:properties rordam} this implies $\overline{E}(f^*f)=0$. Now faithfulness of $\overline{E}$ implies $f=0$ and thus $a=0$.
\end{proof}

\begin{remark}
  If we denote by $\iota_n \colon D_n\rightarrow D_{n+1}$ the inclusion, it follows from the construction of $E_n$ that $\iota_n\circ E_n=E_{n+1}\circ \iota_n$ for all $n \in \NN$.
\end{remark}

\begin{corollary} \label[corollary]{cor:d-n-cartan}
  $\widetilde{D}_n$ is a Cartan subalgebra of $D_n$.
\end{corollary}

To prove that $\widetilde{D}$ is also a Cartan subalgebra of $D$ we use the following result.

\begin{theorem}[{See \cite{barlak_cartan_2020}*{Theorem~3.6}, \cite{li-raad-2023}*{Theorem~2.5}}]\label[theorem]{theorem:inductive limit of cartans}
  Let $(A_n,B_n)$ be Cartan pairs with normalizers $N_n\coloneqq N(B_n\subseteq A_n)$ and faithful conditional expectations $P_n \colon A_n\rightarrow B_n$. Let $\varphi_n \colon A_n\rightarrow A_{n+1}$ be injective $^*$-homomorphisms with $\varphi_n(B_n)\subseteq B_{n+1}$, $\varphi_n(N_n)\subseteq N_{n+1}$ and $P_{n+1}\circ\varphi_n=\varphi_n\circ P_n$ for all $n$. Then $\varinjlim(B_n,\varphi_n)$ is a Cartan subalgebra of $\varinjlim(A_n,\varphi_n)$.
\end{theorem}

\begin{corollary}\label[corollary]{cor:D Cartan}
  $\widetilde{D}$ is a Cartan subalgebra of $D$.
\end{corollary}
\begin{proof}
  We have $D=\varinjlim(D_n,\iota_n)$ and $\widetilde{D}=\varinjlim(\widetilde{D}_n,\iota_n)$ where $\iota_n \colon D_n\rightarrow D_{n+1}$ is the inclusion. We know $(D_n,\widetilde{D}_n)$ is a Cartan pair for all $n$. Furthermore, normalizers of $\widetilde{D}_n$ are mapped to normalizers of $\widetilde{D}_{n+1}$ and $\iota_n\circ E_n=E_{n+1}\circ\iota_n$ for all $n$, so it follows from \Cref{theorem:inductive limit of cartans} that $(D,\widetilde{D})$ is a Cartan pair too.
\end{proof}

By Renault's correspondence in \Cref{theorem:correspondence Renault} and \cref{cor:D Cartan} there are a twisted groupoid $(\CG,\Sigma)$ and an isomorphism $\psi \colon D\rightarrow C^*_r(\CG,\Sigma)$ carrying $\widetilde{D}$ to $C_0(\CGo)$.
We will see in the next chapter how this gives a groupoid model for R\o rdam's algebra $D\rtimes_\alpha\ZZ$. We also know that $(D_n,\widetilde{D}_n)$ is a Cartan pair for all $n\in \NN$, so each $D_n$ also has a twisted groupoid model, which we denote by $C^*_r(\CG_n,\Sigma_n)$.

\section{Further structure}\label[section]{sec:further structure}

This section describes the structure of the Cartan pair $\widetilde{D} \subseteq D$.
We are mostly interested in the unique extension property and the ``untwistedness'' of the groupoid model derived from Renault's correspondence, see \cref{theorem:correspondence Renault}.
This will yield, in turn, a particularly well-behaved model for \(D \rtimes_\alpha \ZZ\).
To conclude, we also examine the nonzero finite projection $Q$ and the infinite projection $P$ in R\o rdam's algebra \(D \rtimes_\alpha \ZZ\).
  
\subsection{Unique extension property and \texorpdfstring{\cstar{}}{C*-}diagonality}\label[subsection]{subsec:unique extension}

We start by proving that the pair $\widetilde{D}_n \subseteq D_n$ has the unique extension property, meaning that each pure state on $\widetilde{D}_n$ has a unique extension to a pure state on $D_n$.
By \cref{cor:d-n-cartan} and the considerations in \cite{li-raad-2023} this implies that the limit $\widetilde{D}$ is also a $C^*$-diagonal.

We start by describing the pure states on $\widetilde{D}_n$, \emph{i.e.}\ its characters.
Notice that a character $\phi \colon \widetilde{A} = C(Z,c_0(\NN)) \to \CC$ is of the form \(\phi(a)=\langle a(x)\delta_k,\delta_k\rangle\) for some $ x\in Z$ and $k\in \NN$,.
Intending to simplify notation, we henceforth denote this character by $\widetilde{\phi}_{x,k}$.
The same formula defines a map $\phi_{x,k} \colon \diagelms\rightarrow\CC$:
\begin{equation}\label{eq:phi_x,k}
  \phi_{x,k}(b)\coloneqq \langle b(x)\delta_k,\delta_k\rangle.
\end{equation}
Since \mbox{$\widetilde{A}_n=\alpha^n\big(\mu_{\infty,0}(\widetilde{A})\big)\cong \widetilde{A}$}, we may construct a character $\phi^{(n)}_{x,k} \colon \widetilde{A}_n\rightarrow \CC$ by defining
\begin{equation*}
    \phi^{(n)}_{x,k}\big(\alpha^n\big(\mu_{\infty,0}(a)\big)\big) \coloneqq \langle a(x)\delta_k,\delta_k\rangle
\end{equation*}
for some $x\in Z$ and $k\in \NN$. As it turns out, these are all the characters of \(\widetilde{D}_n\).

\begin{lemma}\label[lemma]{lemma:character space}
  The character space of $\widetilde{D}_n$ is isomorphic to $(Z\times\NN)\times\{-n,-n+1,\dots,n\}$ for all $n\in \NN$.
\end{lemma}
\begin{proof}
    
  Take $n\in \NN$ and suppose $\phi$ is a character on $\widetilde{D}_n$. From \Cref{lemma:D_n tilde as sum}, we know that \(\widetilde{D}_n=\widetilde{A}_{-n}+\dots+\widetilde{A}_n\).
  Define $\phi_i\coloneqq\phi|_{\widetilde{A}_i}$ for $i=-n,\dots,n$. For each $i$, $\phi_i$ is a multiplicative linear functional $\phi_i \colon \widetilde{A}_i\rightarrow\CC$, so it is either zero or a character on $\widetilde{A}_i$. Take the smallest $i$ for which $\phi_i\neq 0$, we will prove that this $\phi_i$ determines all other $\phi_j$ with $j\geq i$. We have $\phi_i=\phi^{(i)}_{x,k}$ for some $x\in Z$ and $k\in \NN$.
  Moreover, let $\Hat{\overline{\varphi}} \colon Z \times \NN \to Z \times \NN$ be the Gelfand dual of $\overline{\varphi}|_{\widetilde{A}} \colon \widetilde{A} \to \overline{\varphi}(\widetilde{A})$ (recall that \(\widetilde{A} = C(Z, c_0(\NN))\) is commutative and \(\overline{\varphi}\) is an injective \(^*\)-homomorphism).

  Given two arbitrary elements $a,b\in C(Z,c_0(\NN))$, we have that
  \begin{align*}
    \phi\big(\alpha^{i}\big(\mu_{\infty,0}(a)\big) \cdot \alpha^{i+1}\big(\mu_{\infty,0}(b)\big)\big)&=\phi_i\big(\alpha^{i}\big(\mu_{\infty,0}(a \cdot \overline{\varphi}(b))\big)\big) = \langle a(x) \cdot \overline{\varphi}(b)(x)\delta_k,\delta_k\rangle\\
    & \stackrel{{\rm (\cref{lemma:diagonality preserved})}}{=} \langle a(x)\delta_k,\delta_k\rangle \cdot \langle \overline{\varphi}(b)(x)\delta_k,\delta_k\rangle\\
    &=\phi_i\big(\alpha^i\big(\mu_{\infty,0}(a)\big)\big) \cdot \big(\phi_{x,k}\circ \overline{\varphi}\big)(b) =\phi_i\big(\alpha^i\big(\mu_{\infty,0}(a)\big)\big) \cdot \Hat{\overline{\varphi}}(\phi_{x,k})(b),
  \end{align*}
  where in the third equality we also used $\phi_{x,k} \colon \diagelms \rightarrow \CC$ is multiplicative and $\phi_{x,k}|_{\overline{\varphi}(\widetilde{A})} \colon \overline{\varphi}(\widetilde{A}) \to \CC$ is a character.
  We have $\Hat{\overline{\varphi}}(\phi_{x,k})=\phi_{\beta_1(x,k),\beta_2(x,k)}$ for some $\beta_1(x,k)\in Z, \beta_2(x,k)\in \NN$, whence
  \begin{align*}
    \phi_i\big(\alpha^i\big(\mu_{\infty,0}(a)\big)\big) \cdot \Hat{\overline{\varphi}}(\phi_{x,k})(b)&=\phi_i\big(\alpha^i\big(\mu_{\infty,0}(a)\big)\big) \cdot \phi_{\beta_1(x,k),\beta_2(x,k)}(b)\\
    &=\phi_i\big(\alpha^i\big(\mu_{\infty,0}(a)\big)\big) \cdot \phi_{\beta_1(x,k),\beta_2(x,k)}^{(i+1)}\big(\alpha^{i+1}(\mu_{\infty,0}(b))\big).
  \end{align*}
  On the other hand, using multiplicativity of $\phi$ also yields
  \begin{align*}
    \phi\big(\alpha^{i}\big(\mu_{\infty,0}(a)\big) \cdot \alpha^{i+1}\big(\mu_{\infty,0}(b)\big)\big)&=\phi\big(\alpha^{i}\big(\mu_{\infty,0}(a)\big)\big) \cdot \phi\big(\alpha^{i+1}\big(\mu_{\infty,0}(b)\big)\big)\\
    &=\phi_i\big(\alpha^i\big(\mu_{\infty,0}(a)\big)\big) \cdot \phi_{i+1}\big(\alpha^{i+1}\big(\mu_{\infty,0}(b)\big)\big).
  \end{align*}
  Since this holds for all $b\in C(Z,c_0(\NN))$, it follows that $\phi_{i+1}=\phi_{\beta_1(x,k),\beta_2(x,k)}^{(i+1)}$.
  Iterating this procedure yields that $\phi_{i+1}$ determines $\phi_{i+2}$, and thus, in turn, \(\phi_j\) is determined for all \(j \geq i\).
  This precisely means a character $\phi \colon \widetilde{D}_n\rightarrow \CC$ is completely determined by the smallest $i$ for which $\phi_i\neq 0$ and the $x\in Z$ and $k\in \NN$ such that $\phi_i=\phi_{x,k}^{(i)}$. It follows that the character space of $\widetilde{D}_n$ is contained in $(Z\times \NN)\times \{-n,-n+1,\dots,n\}$. 
    
  We turn to prove that the character space is actually equal to $(Z\times\NN)\times\{-n,\dots,n\}$.
  Take $\big((x,k),i\big)\in (Z\times\NN)\times\{-n,-n+1,\dots,n\}$ and define $\phi \colon \widetilde{D}_n\rightarrow\CC$ by $\phi_j \coloneqq 0$ for $j<i$, $\phi_i \coloneqq \phi_{x,k}^{(i)}$ and $\phi_j$ for $j>i$ defined by $\phi_j \coloneqq \phi_{\beta_1^{j-i}(x,k),\beta_2^{j-i}(x,k)}^{(j)}$. 
  The map $\phi$ is nonzero, since $\phi_i$ is nonzero, and it it is linear since each $\phi_j$ is. Take $j\geq i$, so $\phi_j=\phi^{(j)}_{x',k'}$ for some $x'\in Z$ and $k'\in\NN$. By construction and a similar argument to the one in the previous paragraph, we have for all $l$ and all $a,b\in C(Z,c_0(\NN))$,
  \begin{align*}
        \phi_j\big(\alpha^j(\mu_{\infty,0}\big(a \cdot \overline{\varphi}^l(b))\big)\big)&=\langle a(x') \cdot \overline{\varphi}^l(b)(x')\delta_{k'},\delta_{k'}\rangle=\langle a(x')\delta_{k'},\delta_{k'}\rangle \cdot \langle\overline{\varphi}^l(b)(x')\delta_{k'},\delta_{k'}\rangle\\
    &=\phi_j\big(\alpha^j(\mu_{\infty,0}(a))\big) \cdot \phi_{j+l}\big(\alpha^{j+l}(\mu_{\infty,0}(b))\big).
  \end{align*}
  On the other hand, when $j<i$ we have $\phi_j=0$, and thus it is clear that
  \begin{equation*}
    0 = \phi_j\big(\alpha^j(\mu_{\infty,0}\big(a\overline{\varphi}^l(b))\big)\big)=\phi_j(\alpha^j(\mu_{\infty,0}(a)))\phi_{j+l}(\alpha^{j+l}(\mu_{\infty,0}(b))) = 0.
    \end{equation*}
    This implies that the map $\phi$ defined this way is multiplicative, and thus a character on $\widetilde{D}_n$.
\end{proof}

\begin{notation} \label[notation]{notation:phi-xki}
  Following the notation in the proof of \cref{lemma:character space}, the character $\phi \colon \widetilde{D}_n\rightarrow\CC$ with $\phi_j=0$ for $j<i$, $\phi_i=\phi_{x,k}^{(i)}$ and $\phi_j=\phi_{\beta_1^{j-i}(x,k),\beta_2^{j-i}(x,k)}^{(j)}$ for $j>i$ is henceforth denoted by $\phi_{x,k,i} \colon \widetilde{D}_n \to \CC$.
\end{notation}

\begin{proposition} \label[proposition]{lemma:diagonal}
  The Cartan pair $\widetilde{D} \subseteq D$ is a $C^*$-diagonal.
  Hence, its Weyl groupoid is principal.
\end{proposition}
\begin{proof}
  By \cite{li-raad-2023}*{Theorem~2.5}, and using the fact that the pair \(\widetilde{D} \subseteq D\) is (by construction) the inductive limit of the pairs \(\widetilde{D}_n \subseteq D_n\), it suffices to prove the latter pair is a $C^*$-diagonal for all $n \in \NN$.

  Fix \(n \in \NN\), and let $\widetilde{\phi} \colon \widetilde{D}_n \to \CC$ be a pure state.
  We aim to show it extends uniquely to a (pure) state of \(D_n\).
  By \Cref{lemma:character space} $\widetilde{\phi}$ is of the form $\phi_{x,k,i}$ for some $x\in Z$, $k\in\NN$ and $i\in \{-n,\dots ,n\}$ (recall \cref{notation:phi-xki}).
  Let $\phi$ be an extension of $\widetilde{\phi}$ to a state on $D_n=A_{-n}+\dots +A_n$.
  Look at $\phi|_{A_l}$ for some $l\geq i$, so $\phi|_{\widetilde{A}_l}=\widetilde{\phi}|_{\widetilde{A}_l}$ is a character.
  Furthermore, $\phi|_{A_l}$ is still linear and positive and $\norm{\phi|_{A_l}}\leq1$.
  Since $\phi|_{\widetilde{A}_l}$ is a character and thus a pure state on $\widetilde{A}_l$, we also have $ \|\phi|_{\widetilde{A}_l}\|=1$, which implies that $\norm{\phi|_{A_l}}=1$, so $\phi|_{A_l}$ is a state. This state extends the pure state $\phi|_{\widetilde{A}_n}$, so it is uniquely determined since $\widetilde{A}_n$ is a $C^*$-diagonal of $A_n$.

  Now look at $\phi|_{A_l}$ with $l<i$, so $\phi|_{\widetilde{A}_l}=\widetilde{\phi}|_{\widetilde{A}_l}=0$.
  Let $\{e_m\}_{m \in \NN}\subset C(Z,c_0(\NN))$ be the usual approximate unit.
  Put $e_m^{(l)}=\alpha^l\big(\mu_{\infty,0}(e_m)\big)\in \widetilde{A}_l$, now $\{e_m^{(l)}\}_{m\in\NN}$ is an approximate unit of projections for $A_l$. Take $a\in A_l$. We have \(|\phi(e_m^{(l)}a)|^2 \leq \phi(e_m^{(l)})\phi(a^*a)=0\) for all $m$, where the last equality follows from $e_m^{(l)}\in \widetilde{A}_l$ and thus $\phi(e_m^{(l)})=0$. This implies \(\phi(a)=\lim_{m\rightarrow\infty}\phi(e_m^{(l)}a)=0\)
  for all $a\in A_l$. It follows that $\phi|_{A_l}$ is uniquely determined by $\widetilde{\phi}$ for all $l$, and since $D_n=A_{-n}+\dots+A_n$ by \Cref{prop:properties rordam}, this implies that $\phi$ is uniquely determined by $\widetilde{\phi}$. We have now proved that $\widetilde{D}_n$ has the unique extension property, so it is a $C^*$-diagonal. By \cite{li-raad-2023}*{Theorem~2.5}, this implies that the limit $\widetilde{D}$ is also a $C^*$-diagonal, which is what we wanted to prove.

  The ``hence'' statement follows from \cite{kumjian-diagonals-1986}*{Theorem~3.1}, see also \cref{theorem:correspondence Renault}.
\end{proof}
  
\subsection{Triviality of the twist}\label[subsection]{subsec:trivial twist}

We prove that the groupoid model we found for $\widetilde{D} \subseteq D$ in \Cref{sec:Cartan} is actually untwisted by proving the same for every pair $\widetilde{D}_n \subseteq D_n$: that is, we show that the extension \(\CGo_n \times \TT \to \Sigma_n \to \CG_n\) splits.
Recall this precisely means we have to construct a continuous section $s_n \colon \CG_n\rightarrow\Sigma_n$ that is also a groupoid homomorphism, see \cref{prop:section} below.
In order to do this, as usual, we can use that the orignal pair $\widetilde{A} \subseteq A$ has an untwisted groupoid model, and then propagate it.
On the one hand, we keep the same notation for the characters as we did in \Cref{subsec:unique extension}.
This notwithstanding, we need to introduce some extra notation, for we will need to explicitly construct the groupoids \(\CG_n\) and \(\Sigma_n\).

Given some $b\in \widetilde{D}_n$, by \Cref{lemma:D_n tilde as sum} there are some $b_j \in \widetilde{A}_j$ with
\begin{equation*}
  b=\alpha^{-n}\big(\mu_{\infty,0}(b_{-n})\big)+\dots + \alpha^i\big(\mu_{\infty,0}(b_i)\big)+\dots+\alpha^n\big(\mu_{\infty,0}(b_n)\big).
\end{equation*}
Given \(i \in \{-n, \dots, n\}\), we define \(B_i \coloneqq b_i+\dots+\overline{\varphi}^{n-i}(b_n)\).
The construction of the character $\phi_{x,k,i}$ gives
\begin{align}\label{eq:phi_x,k,i}
  \phi_{x,k,i}(b)&=\phi_{x,k,i}\big(\alpha^{-n}\big(\mu_{\infty,0}(b_{-n})\big)\big)+\dots + \phi_{x,k,i}\big(\alpha^i\big(\mu_{\infty,0}(b_i)\big)\big)+\dots+\phi_{x,k,i}\big(\alpha^n\big(\mu_{\infty,0}(b_n)\big)\big)\notag\\
  &=\phi_{x,k}(b_i)+\dots+\phi_{x,k}\big(\overline{\varphi}^{n-i}(b_n)\big)\notag\\
  &=\phi_{x,k}\big(b_i+\dots+\overline{\varphi}^{n-i}(b_n)\big)=\phi_{x,k}(B_i).
\end{align}
Note that every pure state on $\widetilde{D}_n$ is of the form $\phi_{x,k,i}$ for $x\in Z$, $k\in \NN$ and $i\in \{-n,\dots,n\}$.

\vspace{1mm}

We now recall the construction of the Weyl groupoid and twist of a Cartan pair \(B \subseteq A\). The (unique) polar decomposition (which exists in the von Neumann algebra \(A^{**}\)) shows that a normalizer \(m \in N(B \subseteq A)\) defines uniquely a partial homeomorphism \(\theta_m \colon D_{m^*m} \to D_{mm^*}\) between open subsets \(D_{m^*m}, D_{mm^*} \subseteq \widehat{B}\).
Notice we write \(D_{m^*m} \coloneqq \{\phi \in \widehat{B} : \phi(m^*m) > 0\}\) to indicate that it only depends on \(m^*m\), not on \(m\) itself.
In fact, see \cite{sims2018etalegroupoidscalgebras}*{Proposition 5.1.15} (and compare with \cite{kumjian-diagonals-1986}*{1.6} and \cite{Renault2008CartanSI}*{Proposition 4.7}), the defining property of \(\theta_m\) is
\begin{equation} \label{def-trait-theta}
  \theta_m\left(\phi\right)\left(mfm^*\right) = \phi\left(fm^*m\right) \; \text{ for all } f \in C_0\left(D_{m^*m}\right) \text{ and } \phi \in D_{m^*m}.
\end{equation} 
We briefly note in passing that we see \(C_0(D_{m^*m})\) as an ideal in \(C_0(\widehat{B}) \cong B\) in the usual way, that is, we extend a function \(f \in C_0(D_{m^*m})\) by \(0\) outside of the open (not necesssarily closed) set \(D_{m^*m}\).

It follows from the construction of Renault in \cite{Renault2008CartanSI} that the Weyl extension \(\CGo \times \TT \to \Sigma \to \CG\) can be constructed from the collection \(\{\theta_m\}_{m \in N(B \subseteq A)}\).
For instance, \(\Sigma\) is the set \(\{(m, \phi) : \phi \in D_{m^*m}\}\) quotiented out by the equivalence relation given by
\begin{align*}
  \left(m_1, \phi_1\right) \approx \left(m_2, \phi_2\right) \Leftrightarrow \phi_1 = \phi_2 & \text{ and there is some } f \in B \text{ such that } \\
  & m_1 f = m_2 f \text{ and } \phi_1\left(f\right) \neq 0. 
\end{align*}
Likewise, \(\CG\) is the same set \(\{(m, \phi) : \phi \in D_{m^*m}\}\) quotiented out by the (weaker) equivalence relation
\begin{align*}
  \left(m_1, \phi_1\right) \sim \left(m_2, \phi_2\right) \Leftrightarrow \phi_1 = \phi_2 & \text{ and there is some open } \CU \subseteq \widehat{B} \text{ such that } \\
  & \theta_{m_1}(\psi) = \theta_{m_2}(\psi) \text{ for all } \psi \in \CU \text{ and } \phi_1 \in \CU.
\end{align*}
Notice that the extension \(\CGo \times \TT \to \Sigma \to \CG\) is now given by the canonical inclusion and projection maps.
Henceforth, the \(\sim\) and \(\approx\) classes of \((m, \phi)\) will respectively be denoted by \([m, \phi]\) and \(\llbracket m, \phi \rrbracket\).
In particular note that one may also write \([m, \phi]\) more succintly as \([\theta_m, \phi]\), for it only depends on the partial homeomorphism \(\theta_m\). We will refrain from doing so in this paper.
  
\begin{lemma}\label[lemma]{lemma:alpha_m}
  If $m\in N(\widetilde{D}_n\subseteq D_n)$ and $\phi_{x,k,i}\in \operatorname{dom}(m)$, then $\theta_m(\phi_{x,k,i})=\phi_{x,j,i}$ for some $j\in \NN$.
  In fact, there is an open neighborhood $U$ of $x$ such that $\theta_m(\phi_{z,k,i})=\phi_{z,j,i}$ for all $z\in U$.
\end{lemma}
\begin{proof}
  Take $m\in N(\widetilde{D}_n\subseteq D_n)$ and $\phi_{x,k,i}\in \text{dom}(m)$ as in the statement.
  Fix some $b\in \widetilde{D}_n$.
  By \eqref{eq:phi_x,k,i}:
  \begin{equation}\label{eq:phi_x,k,i 1}
    \phi_{x,k,i}(bm^*m)= \phi_{x,k}(B_iM_i^*M_i) = \phi_{x,k}(B_i) \phi_{x,k}(M_i^*M_i),
  \end{equation}
  where \(M_i \coloneqq m_i+\dots+\overline{\varphi}^{n-i}(m_n)\) and \(B_i \coloneqq b_i+\dots+\overline{\varphi}^{n-i}(b_n)\) are as in \eqref{eq:phi_x,k,i}. Since $m$ is a normalizer and $\phi_{x,k,i}(m^*m)>0$, $M_i(x)$ contains a non-zero term of the form $ae_{jk}$ for some $j$ with $a\neq 0$ and no other terms of the form $e_{hk}$ or $e_{jh}$. Using \eqref{eq:phi_x,k,i} again we find
  \begin{equation}\label{eq:phi_x,k,i 2}
    \phi_{x,j,i}(mbm^*)=\phi_{x,j}(M_iB_iM_i^*).
  \end{equation}
  Now using the definition of $\phi_{x,j}$ from \eqref{eq:phi_x,k} and the fact that $M_i(x)$ contains a term of the form $ae_{jk}$ with $a\neq 0$ and no other terms of the form $e_{hk}$ or $e_{jh}$, we find
  \begin{equation}\label{eq:phi_x,k,i 3}
    \phi_{x,j}(M_iB_iM_i^*)=\langle M_iB_iM_i^*(x)\delta_j,\delta_j\rangle=\langle ae_{jk}B_i(x)a^*e_{kj}\delta_j,\delta_j\rangle = |a|^2\langle B_i(x)\delta_k,\delta_k\rangle.
  \end{equation}
  Similarly, \(\phi_{x,k}(M^*M)=|a|^2.\)
  Pairing this last equation with \eqref{eq:phi_x,k,i 2} and \eqref{eq:phi_x,k,i 3} yields
  \begin{equation}\label{eq:phi_x,k,i 5}
    \phi_{x,j,i}(mbm^*)=\phi_{x,k}(B_i)\phi_{x,k}(M_i^*M_i).
  \end{equation}
  Since this holds for all $b\in \widetilde{D}_n$, it follows from \eqref{def-trait-theta}, \eqref{eq:phi_x,k,i 1} and \eqref{eq:phi_x,k,i 5} that $\theta_m(\phi_{x,k,i})=\phi_{x,j,i}$.

  The ``in fact'' statement follows from \Star{}strong continuity of $M_i \in \widetilde{D}_n$.
  Indeed, there is a neighborhood $U$ of $x$ such that $M_i(z)$ contains a term of the form $ae_{jk}$ with $a\neq 0$ and no other terms of the form $e_{hk}$ or $e_{jh}$ for all $z\in U$. This implies that $\phi_{z,k,i}\in \operatorname{dom}(m)$ and $\theta_m(\phi_{z,k,i})=\phi_{z,j,i}$ for all $z\in U$.
\end{proof}
  
Suppose $m\in N(\widetilde{D}_n\subseteq D_n)$ and $\phi_{x,k,i}\in \operatorname{dom}(m)$ such that $\theta_m(\phi_{z,k,i})=\phi_{z,j,i}$ for all $z\in U$ for an open neighborhood $U$ of $x$ (as in \cref{lemma:alpha_m}). If we define $M_i$ as in \eqref{eq:phi_x,k,i} and $p_{kk}\in A$ by $p_{kk}(x)\coloneqq e_{kk}$ for all $x\in Z$, then $M_ip_{kk}\in N(\widetilde{A}\subseteq A)$ is still a normalizer. 
We can compute $\theta_{M_ip_{kk}}(\widetilde{\phi}_{z,k})$ as follows.
  
\begin{lemma}\label[lemma]{lemma:alpha_m tilde}
  $\theta_{M_ip_{kk}}(\widetilde{\phi}_{z,k})=\widetilde{\phi}_{z,j}$ for all $z\in U$.
\end{lemma}
\begin{proof}
  Take $z\in U$ and $b\in \widetilde{A}$, we have
  \begin{align*}
    \widetilde{\phi}_{z,k}(bp_{kk}M_i^*M_ip_{kk})&=\widetilde{\phi}_{z,k}(b)\widetilde{\phi}_{z,k}(p_{kk}M_i^*M_ip_{kk}) = \widetilde{\phi}_{z,k}(b)\widetilde{\phi}_{z,j}(M_ip_{kk}p_{kk}M_i^*)\\
    &= \widetilde{\phi}_{z,j}(M_ip_{kk}bp_{kk}M_i^*).
  \end{align*}
  Since this holds for all $b\in \widetilde{A}$, it follows that $\theta_{M_ip_{kk}}(\widetilde{\phi}_{z,k})=\widetilde{\phi}_{z,j}$ for $z\in U$.
\end{proof}

\begin{convention}
    In the following, since we will be dealing with several Cartan pairs at the same time, we will denote by \(\CG_{(B, A)}\) and \(\Sigma_{(B,A)}\) the Weyl groupoid and twist of a Cartan pair \(B \subseteq A\) respectively.
\end{convention}

By \cite{kumjian-diagonals-1986}*{Remark~4.2}, triviality of the twist \(\Sigma_n\) of \(\widetilde{D}_n \subseteq D_n\) is equivalent to the existence of a continuous section $s_n \colon \CG_n\rightarrow\Sigma_n$ which is, moreover, a groupoid morphism.
By triviality of the twist $\Sigma_{(A,\widetilde{A})}$, there is a continuous section
\begin{equation*}
  s_A \colon \CG_{(A,\widetilde{A})}\rightarrow\Sigma_{(A,\widetilde{A})}
\end{equation*}
that is a groupoid homomorphism.
(Notice that, in this case, we have much better descriptions of \(\CG_{(A,\widetilde{A})}\) and \(\Sigma_{(A,\widetilde{A})}\), but we refrain from using them in pursuit of clarification.)
Take an element in $\CG_n$.
By definition it is of the form $[m,\phi_{x,k,i}]$, where
\begin{equation*}
  m=\alpha^{-n}\big(\mu_{\infty,0}(m_{-n})\big)+\alpha^{-n+1}\big(\mu_{\infty,0}(m_{-n+1})\big)+\dots+\alpha^n\big(\mu_{\infty,0}(m_n)\big)
\end{equation*}
for $m_{-n},\dots,m_n\in A$ and \mbox{$\phi_{x,k,i}(m^*m)>0$.} We can now define 
\begin{equation*}
  M_i \coloneqq m_i+\overline{\varphi}(m_{i+1})+\dots+\overline{\varphi}^{n-i}(m_n)
\end{equation*}
as before. We have $s_A([M_ip_{kk},\widetilde{\phi}_{x,k}])=\llbracket m',\widetilde{\phi}_{x',k'}\rrbracket$ for some normalizer $m'\in N(\widetilde{A}\subseteq A)$ and $(x',k')\in Z\times \NN$. Since $s_A$ is a section, we have $[m',\widetilde{\phi}_{x',k'}]=[M_ip_{kk},\widetilde{\phi}_{x,k}]$ and thus $\widetilde{\phi}_{x',k'}=\widetilde{\phi}_{x,k}$. Hence the following is well posed:
\begin{equation}\label{eq:s_n}
  s_n([m,\phi_{x,k,i}])\coloneqq \llbracket \alpha^i\big(\mu_{\infty,0}(m')\big),\phi_{x,k,i}\rrbracket.
\end{equation}
  
\begin{proposition}\label[proposition]{prop:section}
  $s_n \colon \CG_n\rightarrow\Sigma_n$ is a well-defined, continuous section that is, moreover, a homomorphism.
\end{proposition}
\begin{proof}
  We first prove \(s_n\) is well-defined. Suppose $[m_1,\phi_{x,k,i}]=[m_2,\phi_{x,k,i}]$, that is, there is a neighborhood $U$ of $x$ such that
  \begin{equation}\label{eq:s_n 1}
     \theta_{m_1}(\phi_{z,k,i})=\theta_{m_2}(\phi_{z,k,i})\hspace{5mm}\text{ for all }z\in U. 
  \end{equation}
  From \Cref{lemma:alpha_m}, we know that there is another neighborhood $U'$ of $x$ and a $j\in\NN$ such that
  \begin{equation}\label{eq:s_n 2}
     \theta_{m_1}(\phi_{z,k,i})=\phi_{z,j,i}\hspace{5mm}\text{ for all }z\in U'.
  \end{equation}
  Putting $V \coloneqq U\cap U'$, by \eqref{eq:s_n 1} and \eqref{eq:s_n 2} it follows that
    \begin{equation}\label{eq:s_n 3}
        \theta_{m_2}(\phi_{z,k,i})=\theta_{m_1}(\phi_{z,k,i})=\phi_{z,j,i}\hspace{5mm}\text{ for all }z\in V.
  \end{equation}
  If we define $M_{1,i}$ and $M_{2,i}$ in the same way as before for \(m_1, m_2\) respectively, \eqref{eq:s_n 3} and \Cref{lemma:alpha_m tilde} imply that 
  \begin{equation*}
     \theta_{M_1p_{kk}}(\widetilde{\phi}_{z,k})=\widetilde{\phi}_{z,j}=\theta_{M_2p_{kk}}(\widetilde{\phi}_{z,k})\hspace{5mm}\text{ for all }z\in V.
  \end{equation*}
  In particular, from the way the equivalence relation on $\CG_{(A,\widetilde{A})}$ is defined, this precisely means that \([M_{1,i}p_{kk},\widetilde{\phi}_{x,k}]=[M_{2,i}p_{kk},\widetilde{\phi}_{x,k}]\).
  Now suppose $s_A([M_{1,i}p_{kk},\widetilde{\phi}_{x,k}])=\llbracket m'_1,\phi_{x,k}\rrbracket$ and $s_A([M_{2,i}p_{kk},\widetilde{\phi}_{x,k}])=\llbracket m'_2,\phi_{x,k}\rrbracket$. Since $s_A$ is well-defined, this gives $\llbracket m'_1,\phi_{x,k}\rrbracket=\llbracket m'_2,\phi_{x,k}\rrbracket$. This implies the existence of some $b,b'\in \widetilde{A}$ such that $\widetilde{\phi}_{x,k}(b),\widetilde{\phi}_{x,k}(b')>0$ and $m_1'b=m_2'b'$, which gives
  \begin{equation*}
      \phi_{x,k,i}\big(\alpha^i\big(\mu_{\infty,0}(b)\big)\big),\phi_{x,k,i}\big(\alpha^i\big(\mu_{\infty,0}(b')\big)\big)>0
  \end{equation*}
  and
  \begin{equation*}
    \alpha^i\big(\mu_{\infty,0}(m_1')\big)\alpha^i\big(\mu_{\infty,0}(b)\big)=\alpha^i\big(\mu_{\infty,0}(m'_2)\big)\alpha^i\big(\mu_{\infty,0}(b')\big).
  \end{equation*}
  From this, it follows that \(s_n([m_1,\phi_{x,k,i}])=s_n([m_2,\phi_{x,k,i}])\), \emph{i.e.}\ the map $s_n$ is well-defined. 
    
    Note $s_n$ is continuous since it is the composition of three continuous maps. First, we have a map from $\CG_n$ to $\CG_{(A,\widetilde{A})}$ which sends an element $[m,\phi_{x,k,i}]$ to $[Mp_{kk},\widetilde{\phi}_{x,k}]$. Then we have the continuous section $s_A$. Finally, we have a map from $\Sigma_{(A,\widetilde{A})}$ to $\Sigma_n$ sending an element $\llbracket m',\widetilde{\phi}_{x,k}\rrbracket$ to $\llbracket \alpha^i(\mu_{\infty,0}(m')),\phi_{x,k,i}\rrbracket$. 
    
  Now we prove that $s_n$ is a section, so we have to prove that
  \begin{equation*}
    [\alpha^i\big(\mu_{\infty,0}(m')\big),\phi_{x,k,i}]=[m,\phi_{x,k,i}],
  \end{equation*}
  where $m'$ is as in the definition of $s_n$ in \eqref{eq:s_n}. Using that $s_A$ is a section, we know that $[m',\widetilde{\phi}_{x,k}]=[Mp_{kk},\widetilde{\phi}_{x,k}]$. This implies that there is a neighborhood $U$ of $x$ such that $\theta_{Mp_{kk}}(\phi_{z,k})=\theta_{m'}(\phi_{z,k})$ for all $z\in U$.
  As before, we can find a neighborhood $V$ of $x$ and $j\in \NN$ such that $\theta_{Mp_{kk}}(\phi_{z,k})=\theta_{m'}(\phi_{z,k})=\widetilde{\phi}_{z,j}$ for all $z\in V$. For $z\in V$, we have
  \begin{equation*}
    \theta_{\alpha^i\big(\mu_{\infty,0}(m')\big)}(\phi_{z,k,i})=\phi_{z,j,i}=\theta_m(\phi_{z,k,i})
  \end{equation*}
  and thus $[\alpha^i\big(\mu_{\infty,0}(m')\big),\phi_{x,k,i}]=[m,\phi_{x,k,i}]$, which is what we wanted to prove.
    
  Finally, we have to prove that $s_n$ is a groupoid homomorphism.
  We will first show it sends composable pairs to composable pairs. Suppose $r([m_1,\phi_{x,k,i}])=s([m_2,\phi_{x',k',i'}])$, so $[m_1m^*_1,\theta_{m_1}(\phi_{x,k,i})]=[m_2^*m_2,\phi_{x',k',i'}]$.
  This is equivalent to $\theta_{m_1}(\phi_{x,k,i})=\phi_{x',k',i'}$.
  \Cref{lemma:alpha_m} implies $x'=x$ and $i'=i$. Likewise, it follows from \Cref{lemma:alpha_m tilde} that $\theta_{M_1p_{kk}}(\widetilde{\phi}_{x,k})=\phi_{x,k'}$ and thus 
  \begin{equation*}
    r([M_1p_{kk},\widetilde{\phi}_{x,k}])=s([M_2p_{kk},\widetilde{\phi}_{x,k'}]).
  \end{equation*}
  Suppose $s_A([M_1p_{kk},\widetilde{\phi}_{x,k}])=\llbracket m'_1,\widetilde{\phi}_{x,k}\rrbracket$ and $s_A([M_2p_{kk},\widetilde{\phi}_{x,k}])=\llbracket m'_2,\widetilde{\phi}_{x,k}\rrbracket$. Since $s_A$ is a groupoid homomorphism, this implies \(r(\llbracket m'_1,\widetilde{\phi}_{x,k}\rrbracket)=s(\llbracket m'_2,\widetilde{\phi}_{x,k}\rrbracket)\), which (by the same tokens as before) is equivalent to $\theta_{m'_1}(\widetilde{\phi}_{x,k})=\widetilde{\phi}_{x,k'}$. It follows that $\theta_{\alpha^i\big(\mu_{\infty,0}(m'_1)\big)}(\phi_{x,k,i})=\phi_{x,k',i}$ and thus
  \begin{equation*}
    r(\llbracket \alpha^i\big(\mu_{\infty,0}(m'_1)\big),\phi_{x,k,i}\rrbracket)=s(\llbracket \alpha^i\big(\mu_{\infty,0}(m'_2)\big),\phi_{x,k',i}\rrbracket),
  \end{equation*}
  as desired.
  The fact that $s_n$ respects multiplication is now easy:
  \begin{align*}
    s_n([m_1,\alpha_{m_1}(\phi_{x,k,i})][m_2,\phi_{x,k,i}])&=s_n([m_1m_2,\phi_{x,k,i}])= \llbracket \alpha^i\big(\mu_{\infty,0}(m'_1m'_2)\big),\phi_{x,k,i}\rrbracket\\
    &=s_n([m_1,\alpha_{m_1}(\phi_{x,k,i})])s_n([m_2,\phi_{x,k,i}]),
  \end{align*}
  where we used in the second step that the map $s_A$ is a groupoid homomorphism.

\end{proof}
  
\begin{corollary}\label[corollary]{cor:trivial twist}
  The twist $\CGo \times \TT \to \Sigma \to \CG$ is trivial.
\end{corollary}
\begin{proof}
  By \cite{kumjian-diagonals-1986}*{Remark~4.2} and \Cref{prop:section} the twists $\CGo_n \times \TT \to \Sigma_n \to \CG_n$ at the finite stages are all trivial, whence the claim follows from \cite{barlak_cartan_2020}*{Remark~3.7}.
  For the convenience of the reader, nevertheless, we remind that if \(s_n \colon \CG_n \to \Sigma_n\) is the section constructed in \Cref{prop:section}, then the map
  \[
    \CG_n \times \TT \to \Sigma_n, \;\; \left(g, z\right) \mapsto s_n\left(g\right) \cdot z
  \]
  is a groupoid isomorphism, indeed showing that \(\Sigma_n \cong \CG_n \times \TT\) is trivial.
\end{proof}
  
\subsection{Groupoid model for R\o rdam's algebra and finite/infinite projections}\label[subsection]{subsec:groupoid model rordam}
In this last subsection of the paper we finally construct a groupoid model for \(D \rtimes_\alpha \ZZ\).
We also discuss the finite/infinite projections it contains.
With all the work done up until now these are routine checks, but ones we should make.

By \Cref{cor:trivial twist} the \cstar{}algebra $D$ has an untwisted groupoid model.
This means there is an isomorphism \mbox{$\psi \colon D \rightarrow C^*_r(\CG)$} carrying $\widetilde{D}$ to $C_0(\CGo)$, where \(\CG\) is some locally compact, \'etale, Hausdorff, (necessarily second countable) groupoid.
Nonetheless, we want to find a groupoid model for R\o rdam's algebra $D\rtimes_\alpha\ZZ$. In order to achieve this, we define the crossed product groupoid $\CG\rtimes_{\widetilde{\alpha}}\ZZ$ via:
\begin{align*}
  \widetilde{\alpha} \colon \CG & \to \CG, \\
  [n,\phi] & \to [\alpha(n),\widehat{\alpha}^{-1}(\phi)],
\end{align*}
where \(\widehat{\alpha}\) is the Gelfand dual of the automorphism \(\alpha \colon \widetilde{D} \to \widetilde{D}\).
This is well-defined because $\alpha$ preserves diagonal elements and normalizers, see \Cref{lemma:diagonality preserved,lemma:regularity1}.
For the convenience of the reader, we recall that the crossed product groupoid is defined by the operations
\begin{equation*}
  (\gamma,k)(\beta,l) \coloneqq (\gamma\widetilde{\alpha}^k(\beta),k+l)\hspace{5mm}\text{ and }\hspace{5mm}(\gamma,k)^{-1} \coloneqq (\widetilde{\alpha}^{-k}(\gamma^{-1}),-k).
\end{equation*}
Notice it follows \(r(\gamma,k) = r(\gamma)\) and \(s(\gamma,k) = \widetilde{\alpha}^k(s(\gamma))\).
We may also use the isomorphism $\psi \colon D \to C^*_r(\CG)$ to define an isomorphism \(\psi_\alpha \colon D\rtimes_\alpha\ZZ\rightarrow C^*_r(\CG\rtimes_{\widetilde{\alpha}}\ZZ)\).
Again, this is well known, so we only sketch it.
Define $\psi_\alpha (d u_k) \coloneqq \psi(d)\chi_k$, where $\psi(d)\chi_k \colon \CG\rtimes_{\widetilde{\alpha}}\ZZ\rightarrow\CC$ is the function given by
\begin{equation*}
  (\psi(d)\chi_k)(g,l)=\begin{cases}
        \psi(d)(g)&\text{if } l=k,\\
    0&\text{otherwise}.
  \end{cases}
\end{equation*}
Ultimately because \(\ZZ\) is amenable, the map \(\psi_\alpha\) thus defined extends to a \Star{}homomorphism, which we denote with the same letter by abuse of notation.
Moreover, it is injective, for R\o rdam's algebra is simple. Furthermore, the original $\psi$ is surjective, and thus it is clear that $C_c(\CG\rtimes_{\widetilde{\alpha}}\ZZ)$ is contained in the image of $\psi_\alpha$.
This implies that $\psi_\alpha$ is surjective too, giving the following.

\begin{corollary} \label[corollary]{cor:cartan-model}
    \(\CG \rtimes_{\widetilde{\alpha}} \ZZ\) is a (necessarily minimal) topologically principal groupoid model for \(D \rtimes_\alpha \ZZ\).
    In particular, the latter has a Cartan subalgebra.
\end{corollary}

The proof of \cref{cor:cartan-model} needs one last ingredient. We isolate it due to the fact that we have not been able to find a proper reference in the literature for it. Thus, we write it down for the sake of completion.
It, however, has nothing to do with our setting: it is a purely groupoid behaviour, and probably known to experts.
In the second item below we write \({\rm Orb}_\CG(x) = \{\gamma \gamma^{-1} : \gamma^{-1} \gamma = x\}\) for the orbit of \(x \in \CGo\).

\begin{proposition} \label[proposition]{prop:top-principal}
    Let \(\CG\) be a second countable, locally compact, \'etale, principal groupoid.
    Let \(\varphi \colon \CG \to \CG\) be a groupoid isomorphism.
    The following are equivalent.
    \begin{enumerate}
        \item \label{prop:top-principal:1} \(\CG \rtimes_\varphi \ZZ\) is topologically principal.
        \item \label{prop:top-principal:2} \(\{x \in \CGo : \varphi^n(x) \not\in {\rm Orb}_\CG(x) \text{ for every } n \neq 0\}\) is dense in \(\CGo\).
        \item \label{prop:top-principal:3} Every power \(\widetilde{\varphi}^n \colon C^*_r(\CG) \to C^*_r(\CG)\) is properly outer, where \(n \neq 0\).
    \end{enumerate}
\end{proposition}
\begin{proof}
    The construction and groupoid operations in \(\CG \rtimes_\varphi \ZZ\) were laid out before. It follows that the unit space of \(\CG \rtimes_\varphi \ZZ\) is isomorphic to \(\CGo\).
    Moreover, by definition \(\CG \rtimes_\varphi \ZZ\) is topologically principal if and only if the set of units \(x \in \CGo \cong (\CG \rtimes_\varphi \ZZ)^{(0)}\) with trivial isotropy is dense, and these are precisely the units in item \cref{prop:top-principal:2}. Thus items \cref{prop:top-principal:1,prop:top-principal:2} are indeed equivalent.

    Before turning to the implications regarding \cref{prop:top-principal:3}, recall that by \cite{kwasniewski-meyer-2018}*{Theorem 2.13}, and using that \(\CG\) is second countable, we have that a \Star{}isomorphism \(\widetilde{\varphi}^n \colon C_r^*(\CG) \to C_r^*(\CG)\) is properly outer if and only if it satisfies \emph{Kishimoto's condition}, that is, for all \(g \in C^*_r(\CG)\) and every hereditary subalgebra \(B \subseteq C^*_r(\CG)\) we have that
    \[
        \inf_{f \in B_+, \norm{f} = 1} \norm{f g \widetilde{\varphi}^n\left(f\right)} = 0.
    \]
    Hence we will discuss item \cref{prop:top-principal:3} but with Kishimoto's condition as opposed to proper outerness.

    We first show that \cref{prop:top-principal:3} implies \cref{prop:top-principal:2} or, rather, not \cref{prop:top-principal:2} implies not \cref{prop:top-principal:3}.
    Assume \cref{prop:top-principal:2} fails.
    Let \(\CU \subseteq \CGo\) be some non-empty open set such that for all \(\CU \cap C = \emptyset\), where \(C \subseteq \CGo\) is the set appearing in \cref{prop:top-principal:2}.
    In particular, for all \(x \in \CU\) there is some non-trivial \(n(x) \in \ZZ\) such that \(\widetilde{\varphi}^{n(x)}(x) \in {\rm Orb}_\CG(x)\).
    Hence
    \[
        \CU = \bigcup_{\ell \in \ZZ \setminus \left\{0\right\}} \left\{x \in \CU : n\left(x\right) = \ell\right\}.
    \]
    The latter is a cover of \(\CU\) by closed subsets. By Baire's Category Theorem there is some \(\ell_0 \in \ZZ \setminus \{0\}\) such that \(\{x \in \CU : n(x) = \ell_0\}\) has non-empty interior.
    Whence, without loss of generality, we may assume that \(n(x) = \ell_0\) for all \(x \in \CU\).
    This precisely means that \(\varphi^{\ell_0}(x) \in {\rm Orb}_\CG(x)\) for all \(x \in \CU\), \emph{i.e.}\ there is some open bisection \(v \subseteq \CG\) such that \(\varphi^{\ell_0}(x) = vxv^*\) and \(\CU \subseteq v^*v\).
    Moreover, by further shrinking \(\CU\) if necessary we may assume that \(\overline{\CU} \subseteq v^*v\).
    Consider the hereditary sub-algebra \(B \coloneqq C^*_r(\CG|_{\CU}) \subseteq C^*_r(\CG)\), and let \(g \in C_0(v) \subseteq C_r^*(\CG)\) be any function that is \(1\) on \(v \CU\), which is a non-empty open set. Notice that such \(g\) exists by Urysohn's Lemma.
    It follows that
    \[
        \inf_{f \in B_+, \norm{f} = 1} \norm{f g^* \widetilde{\varphi}^{\ell_0}\left(f\right)} = \inf_{f \in C^*_r(\CG|_{v^*v})_+, \norm{f} = 1} \norm{f g^* g f g^*} = \inf_{f \in C^*_r(\CG|_{v^*v})_+, \norm{f} = 1} \norm{f^2} = 1,
    \]
    contradicting \cref{prop:top-principal:3}.
    Thus, in order to complete the proof it suffices to show that \cref{prop:top-principal:1} implies \cref{prop:top-principal:3}.
    By \cite{armstrong-uniqueness-2022} and \cref{prop:top-principal:1} it follows that \(C_0(\CGo)\) detects ideals in \(C^*_r(\CG \rtimes_{\varphi} \ZZ)\). The latter property is equivalent to the action \(\ZZ \curvearrowright C^*_r(\CG)\) being topologically principal by \cite{kwasniewski-meyer-2018}*{Theorems 2.13 and 2.19}.
\end{proof}

\begin{remark}
    Observe that the implications \cref{prop:top-principal:1} \(\Leftrightarrow\) \cref{prop:top-principal:2} \(\Rightarrow\) \cref{prop:top-principal:3} in \cref{prop:top-principal} are still valid for any action \(\varphi \colon \Gamma \to {\rm Aut}(\CG)\) of a discrete group \(\Gamma\).
\end{remark}

\begin{proof} [Proof of \cref{cor:cartan-model}]
    The fact that \(C^*_r(\CG \rtimes_{\widetilde{\alpha}} \ZZ) \cong D \rtimes_\alpha \ZZ\) has already been discussed.
    Hence, all we have to do is to show that \(\CG \rtimes_{\widetilde{\alpha}} \ZZ\) is topologically principal, which follows from \cref{prop:top-principal} and \cite{rordamsalgebra}*{Lemma 6.9}.
    Whence, a posteriori \(\widetilde{D} \subseteq D \rtimes_\alpha \ZZ\) is maximal abelian.
    Lastly, notice \(\CG \rtimes_{\widetilde{\alpha}} \ZZ\) is minimal.
    Indeed, recall R\o rdam proves in \cite{rordamsalgebra} that \(D \rtimes_\alpha \ZZ\) is simple, and thus, by \cite{armstrong-uniqueness-2022}*{Corollary 6.10}, it follows that \(\CG \rtimes_{\widetilde{\alpha}} \ZZ\) must be minimal, for it is topologically principal.
\end{proof}

The most interesting elements in R\o rdam's algebra are the nonzero finite projection $Q=\mu_{\infty,0}(p_1)$ and the infinite projection $P=\mu_{\infty,0}(g)$ for some constant one-dimensional projection $g$ in $A$. Recall that $p_1\in A$ is defined by
\begin{equation}\label{eq:p_1}
  p_1(x)\coloneqq p(x_1)=\frac{1}{1+|x_1|^2}\begin{pmatrix}
    |x_1|^2&x_1\\
    \overline{x_1}&1
    \end{pmatrix}\in \CK
\end{equation}
for $x\in Z$. We observe the following.
  
\begin{proposition}\label[proposition]{prop:projections Cartan}
  The nonzero finite projection $Q$ does not belong to the Cartan subalgebra $\widetilde{D}$, while the infinite projection $P$ may do.
\end{proposition}
\begin{proof}
  From \eqref{eq:p_1} it is clear that $p_1(x)$ is not diagonal for $x\in Z$, so $p_1\notin \widetilde{A}$. This implies that $Q=\mu_{\infty,0}(p_1) \not\in \widetilde{D}$.
  On the other hand, $g$ is a constant map from $Z$ to some one-dimensional projection $\widetilde{g}$ in $\CK$. If we choose $\widetilde{g}$ to be the projection onto the linear span of some $\delta_k$, then $\widetilde{g}$ is diagonal, which implies that $g\in \widetilde{A}$ and thus $P\in \widetilde{D}_0\subseteq \widetilde{D}$. If we choose $\widetilde{g}$ differently, it may no longer be diagonal, in which case $P\notin\widetilde{D}$.
\end{proof}

Lastly, in \cite{rordamsalgebra}*{Corollary~7.1} R\o rdam proves that $P(D\rtimes_\alpha\ZZ)P$ is a unital, separable, simple, nuclear \cstar{}algebra satisfying the UCT. Furthermore, he proves that $P(D\rtimes_\alpha\ZZ)P$ is infinite, but not purely infinite. Since the \cstar{}algebra is infinite, it is certainly not stably finite, so this \cstar{}algebra gives a counterexample for \Cref{question: intro 2}.
The following is thus our main result, and follows from \Cref{prop:projections Cartan} and \cite{rordamsalgebra}*{Corollary~7.1}.
\begin{corollary}[see \cref{thm-intro,cor-intro}] \label[corollary]{cor:thm-intro}
    There is a locally compact, \'etale, Hausdorff, second countable, amenable, minimal, and topologically principal groupoid \(\CH\) with compact unit space such that the unit \(1 \in C^*_r(\CH)\) is infinite, but \(C^*_r(\CH)\) also contains a non-zero finite projection.

    In particular, \(\CH\) does not have comparison in the sense of \cite{ma-comparison-2020}*{Definition 3.4}.
\end{corollary}

\begin{proof}
    This boils down to finding a groupoid model for corner $P(D\rtimes_\alpha\ZZ)P$.
    We know that $\widetilde{D}\cong C_0(\CGo)$.
    Should we choose the projection $g$ in such a way that $P\in \widetilde{D}$, then $P = \chi_\CU$ is the characteristic function of some clopen $\CU \subseteq \CGo$.
    We may define \(\CH\) as the restricted groupoid $(\CG\rtimes_{\widetilde{\alpha}}\ZZ)|_\CU$:
    \begin{align*}
        (\CG\rtimes_{\widetilde{\alpha}}\ZZ)|_\CU & \coloneqq \{(\gamma,k)\in\CG\rtimes_{\widetilde{\alpha}}\ZZ : r(\gamma)\in \CU \text{ and } \widetilde{\alpha}^k(s(\gamma))\in \CU\}.
    \end{align*}
    Then it follows that \(P(D\rtimes_\alpha\ZZ)P\cong C^*_r\big((\CG\rtimes_{\widetilde{\alpha}}\ZZ)|_\CU\big)\), as desired.
    Notice that some of the properties enjoyed by \(\CH \coloneqq (\CG \rtimes_{\widetilde{\alpha}} \ZZ)|_\CU\) appearing in the statement follow from R\o rdam's \cite{rordamsalgebra}*{Corollary~7.1}.
    For instance, as \(P(D \rtimes_\alpha \ZZ) P\) is separable, \(\CH\) must be second countable. Our previous discussions give that \(\CH\) is locally compact, \'etale and Hausdorff.
    Its unit space is \(\CU\), which is closed in a locally compact, Hausdorff space (namely \(\CGo\)). Hence \(\CH^{(0)} \cong \CU\) is compact.
    Moreover, \(\CH\) must be amenable, see \cite{AD-Renault:amenable-groupoids} (or \cite{buss-martinez-ess-ame-2024}*{Theorem 5.12} for a more recent proof), and also minimal \cite{armstrong-uniqueness-2022}*{Corollary 6.10} and topologically principal, see \cref{cor:cartan-model}.
    The statements about the projections are precisely the point of \cite{rordamsalgebra}*{Corollary~7.1}.
    Furthermore, the ``in particular'' statement follows from \cite{ma-comparison-2020}*{Theorem B} and the fact that \(C^*_r(\CH) \cong P(D \rtimes_\alpha \ZZ)P\) does not have strict comparison of positive elements.
\end{proof}

\bibliography{bibgrpd}

\end{document}